\documentclass[11pt]{amsart}
\usepackage[utf8]{inputenc}
\usepackage[T1]{fontenc}
\usepackage[english,frenchb]{babel}
\AtBeginDocument{\selectlanguage{english}}
\usepackage{appendix}
\usepackage{setspace}
\usepackage{amsmath,amsthm,amsfonts,bm}
\numberwithin{equation}{section}
\usepackage{mathrsfs,amssymb}
\usepackage{mathabx}

\usepackage{enumerate}
\usepackage{verbatim}
\usepackage{mathtools}
\usepackage{cite}
\usepackage{dsfont}
\usepackage{upgreek}
\usepackage{comment}
\usepackage[colorinlistoftodos]{todonotes}
\usepackage{amscd,tikz,tikz-cd} 
\usepackage[all]{xy}
\usepackage{dynkin-diagrams}
\usetikzlibrary{backgrounds}
\usetikzlibrary{matrix,arrows,decorations.pathmorphing}
\usepackage{rank-2-roots} 

\usepackage{xparse}
\let\realItem\item 
\makeatletter
\NewDocumentCommand\myItem{ o }{%
   \IfNoValueTF{#1}%
      {\realItem}
      {\realItem[#1]\def\@currentlabel{#1}}
}
\makeatother

\usepackage{float}

\usepackage{hyperref}
\hypersetup{hypertex=true,
  colorlinks=true,
  linkcolor=blue,
  anchorcolor=blue,
  citecolor=blue}

\usepackage{xcolor}
\usepackage{url}

\hypersetup{
    colorlinks=true,
    urlcolor=blue 
}

\allowdisplaybreaks[1]

\usepackage{calc}

\theoremstyle{plain}
\newtheorem{thm}{Theorem}[section]
\newtheorem*{thm*}{Theorem}

\newtheorem{lem}[thm]{Lemma}
\newtheorem{prop}[thm]{Proposition}
\newtheorem{cor}[thm]{Corollary}

\newtheorem{defn-exmp}[thm]{Definition-Example}

\theoremstyle{definition}
\newtheorem{defn}[thm]{Definition}
\newtheorem{defn-prop}[thm]{Definition-Proposition}

\newtheorem{assmp}[thm]{Assumption}

\theoremstyle{remark}
\newtheorem*{rem}{Remark}

\renewcommand{\refeq}[1]{(\ref{#1})}

\newcommand{\Z}{\mathbb{Z}}
\newcommand{\C}{\mathbb{C}}

\newcommand{\Res}{\operatorname{Res}}

\newcommand{\Fq}{\mathbb{F}_{q}}
\newcommand{\Fpb}{\overline{\mathbb{F}_{p}}}

\newcommand{\End}{\operatorname{End}}

\renewcommand{\Im}{\operatorname{Im}}
\newcommand{\Ker}{\operatorname{Ker}}
\newcommand{\Hom}{\operatorname{Hom}}
\renewcommand{\dim}{\operatorname{dim}}
\newcommand{\Tr}{\operatorname{Tr}}

\newcommand{\Aut}{\operatorname{Aut}}

\newcommand{\pr}{\operatorname{pr}}

\newcommand{\Rep}{\operatorname{Rep}}
\newcommand{\Reps}{\operatorname{Rep}^{\mathfrak{s}}}
\newcommand{\Repf}{\operatorname{Rep}_{\operatorname{f}}}
\newcommand{\Mod}{\operatorname{Mod}}
\newcommand{\Modb}{\operatorname{Mod}\textrm{-}}

\newcommand{\smod}{\textrm{-}\operatorname{mod}}

\newcommand{\supp}{\operatorname{supp}}

\renewcommand{\mod}{\ \operatorname{mod}}

\renewcommand{\H}{\mathcal{H}}

\newcommand{\wfin}{w_{\operatorname{fin}}}

\newcommand{\dom}{\operatorname{dom}}
\newcommand{\TT}{\widebar{T}}
\newcommand{\Ind}{\operatorname{Ind}}

\newcommand{\ind}{\operatorname{ind}}
\newcommand{\res}{\operatorname{res}}

\newcommand{\Waff}{W_{\operatorname{aff}}}
\newcommand{\Saff}{S_{\operatorname{aff}}}
\newcommand{\R}{\mathcal{R}}
\newcommand{\WR}{W(\R)}
\newcommand{\NI}{\operatorname{i}}
\newcommand{\Nr}{\operatorname{r}}

\newcommand{\ND}{\operatorname{D}}

\newcommand{\GL}{\mathrm{GL}}

\newcommand{\ad}{\operatorname{ad}}

\newcommand{\BB}{\mathbf{B}}
\newcommand{\GG}{\mathbf{G}}
\newcommand{\PP}{\mathbf{P}}
\newcommand{\bTT}{\mathbf{T}}
\newcommand{\MM}{\mathbf{M}}

\newcommand{\UU}{\mathbf{U}}
\newcommand{\bLL}{\mathbf{L}}
\newcommand{\bQQ}{{\mathbf{Q}}}

\newcommand{\Lef}{\mathcal{X}}
\newcommand{\Sgn}{\operatorname{sgn}}

\newcommand{\Fr}{\operatorname{F}}
\newcommand{\xHC}[3]{{}^{*}R^{#1}_{#2 \subset #3}}
\newcommand{\HC}[3]{R^{#1}_{#2 \subset #3}}

\newcommand{\Irr}{\operatorname{Irr}}
\newcommand{\Cusp}{\operatorname{Cusp}}

\usepackage[utf8]{inputenc}
\title[An involution for Hecke algebras]{An involution for Hecke algebras}
\author{Chuan Qin}

\calclayout

\begin{document}
\begin{abstract}
We give two generalizations of the Alvis-Curtis duality for Hecke algebras:  an unequal parameter version for the affine Hecke algebras, based on S-I. Kato’s work and a relative version for finite Hecke algebras, based on Howlett-Lehrer’s work.  Our results for the finite case focus on the involution theorem for finite Hecke algebras that appear in Howlett-Lehrer’s theory where they proved a version for characters of certain subgroups of a Weyl group. We hope that our results will serve as a stepping stone for the study of involution for an arbitrary Bernstein block in the $p$-adic reductive group case. We also prove their compatibility with the Alvis-Curtis-Kawanaka duality (Aubert-Zelevinsky duality) when restricted to some Harish-Chandra series \emph{(resp.} Bernstein blocks). This article is part of the author's PhD thesis.
\end{abstract}

\maketitle
\tableofcontents
\section{Introduction}
\subsection{History and background}
  The study of duality for representations with complex coefficients of
  finite Weyl groups and finite groups of Lie type can be traced back from the
  1960s, with contributions from the works
  of D. Alvis \cite{Alvis1979}, L. Solomon \cite{Solomon1966}, and extending to C.W. Curtis
  \cite{Curtis80}, R.B. Howlett \& G.I. Lehrer \cite{HowlettLehrer1982} and P. Deligne \&
  G. Lusztig \cite{DeligneLusztig1982} in 1980s. Several years later, influenced by the proof of \cite{DeligneLusztig1982},  S-I. Kato \cite{Kato1993}  defined a duality operator for representations of Hecke algebras
  associated with a finite or an affine Weyl group with equal parameters, and he
  expressed this operator in terms of a certain involution on the Hecke algebra. Inspired by the
  Hecke algebra case, A-M. Aubert defined in 1995  \cite{Aubert1995} a
  duality (now known as Aubert-Zelevinsky duality) on the Grothendieck group of the category of smooth, finite length,
  complex representations of a $p$-adic group, and she proved several properties
  of this duality.

\subsection{Main results} This article is divided into two parts. The first part is a generalization of
S-I. Kato's results to the unequal parameter case, furthermore along the way we
add more details to the proofs and correct some errors. The second part starts
with the involution theorem for finite Hecke algebras that appear in Howlett-Lehrer's
theory, which serves as a stepping stone for the study of involution for an
arbitrary Bernstein block. Inspired by the finite Hecke algebra case, we obtain the left hand side of involution for the modules of Hecke algebras associated with an arbitrary Bernstein block.

\subsubsection{Involution for unequal parameter generalized affine Hecke algebras}
We now present a more detailed introduction for the first part. Let $\R:= (X, R, Y, R^{\vee})$ be a root datum satisfying axioms listed in
 \cite[Section 1.3]{Solleveld2021} and $\Delta$ denote a set of simple roots of
 $R$. We define $s_{\alpha}$ as the reflection with respect to the hyperplane orthogonal to $\alpha$. Let $W$ be a finite Weyl group generated by the set of simple reflections $S= \left\{
   s_{\alpha}: \alpha
   \in \Delta \right\}$. For any element $w \in W$, we introduce its length $ \ell (w)$ as
 the smallest number  such that $w$ can be written as a product of $ \ell (w)$
 simple reflections, and $ \operatorname{sgn} (w): = (-1)^{\ell (w)}$. Let $w_0$
 denote the
 longest element of $W$. For $I \subset S$ a subset of $S$, let $|I|$ denote the cardinality of
$I$. Let $W_I$ denote the subgroup of $W$ generated by $I$. Also let $\Res_{W_I}^W$
(\emph{resp.} $\Ind_{W_I}^W$) denote the restriction (\emph{resp.} induction functor) for
representations of the finite groups $W$ and $W_I$. In \cite[Theorem
2]{Solomon1966}, L. Solomon proves that for any character $\chi$ of $W$, one has:

  \begin{equation}
    \label{eq:SolomonthmEn}
    \sum_{I \subset S} (-1)^{|I|} \Ind_{W_I}^W \Res_{W_I}^W (\chi) = \widehat{\chi}:= \chi \otimes \operatorname{sgn}.
  \end{equation}
From the perspective that the Hecke algebras can be regarded as a
``deformed'' group algebras associated with Weyl groups, it is natural to wonder
whether there exists a similar result for Hecke algebras modules.

We now introduce the setup for affine Weyl groups and affine Hecke algebras.
From now on, we assume that $R$ is irreducible for simplicity. Let
us define the affine Weyl group $\Waff = W \ltimes \Z R$ (\emph{resp.} extended
affine Weyl group $W(\R) =W \ltimes X$) as the semi-direct product of finite
Weyl group $W$ by the lattice $\Z R$ (\emph{resp.} $X$). Let $w
=\wfin t_{\lambda}$ be the decomposition for $w$ as an element in $ W(\R)$. Here
$\wfin$ denotes the finite Weyl group part and $t_{\lambda}$, $\lambda \in X$ denotes the
translation part
associated with some $\lambda \in X$. We know that $(\Waff, \Saff)$ is a Coxeter
system with $\Waff$
and the set of generators $\Saff := S \bigcup
\ \left\{  \textrm{an affine reflection }s_0 \right\}$. There exists a semi-direct
product decomposition of $W(\R)$ as $\Omega \ltimes
\Waff$ where $\Omega$ is the stabilizer of ``fundamental alcove'' determined by $\Saff$. Let $\ell$ be the length
function for the Coxeter system $(\Waff, \Saff)$: it extends naturally to
$W(\R)$ by requiring $\ell (w) =0$ if $w \in \Omega$. We may thus write $w =  \gamma
\tau$, with $\gamma \in \Omega$ of length $0$ and $\tau$ a product of $\ell(w)$
elements of $\Saff$.

We denote by $\H$ as the extended affine Hecke algebra $\H (W(\R),q_s)$ over a fixed field $K$ associated with $W(\R)$
(see Definition \ref{defn:HeckeIM} for more details) and unequal parameters $q_s \in K^{\times}
\backslash \{ \textrm{roots of unity}\}$, $s \in S$. We also let $\H_{I}$ be the
subalgebra of $\H$ associated with $W(\R)_I: = W_I \ltimes X$.

We define the dual of a $\H$-module $(\pi, M)$ by
\[
  D[M]: = \sum_{I \subset S} (-1)^{|I|}[\Ind_I (\Res_I M)],
\]
where  $\Res_I$ is the usual restriction functor to $\H_I$, while the induction functor
$\Ind_I $ is defined by $\Ind_I N: =\H \otimes_{\H_{I}}N$ for an $\H_I$-module
$N$ (see Definition \ref{defn:indHecke} for more details) and $[M]$ denotes the
image of $M$ in the Grothendieck
group of finite dimensional modules over $K$.

Our first principal theorem is the analogue of the formula \eqref{eq:SolomonthmEn} for
the $\H$-modules:
\begin{thm}\label{thm:KatoEn}
  Define a twisted action $*$ on $\H$ as follows: for \[
    w = \wfin t_{\lambda}= w_{\Omega} t_{\mu} s_{i_1}
    s_{i_2} \cdots s_{i_r} \in \WR
  \]
set
\[
  T_w^{*} = (-1)^{l( \wfin
    )} q(w) T_{w^{-1}}^{-1}
\]
where  $q(w) = \displaystyle \prod_{k=1}^r q_{s_{i_k}}$. Given $\H$-module $(\pi, M)$, let
$(\pi^{*},M^{*})$ be the $\H$-module such that $M^{*} = M$ as $K$-vector space, equipped
with the $\H$-action $\pi^{*}(h)(m) := \pi (h^{*}) (m)$, $\forall m \in M$ and
$\forall h \in \H$. Then we have the following equality in the Grothendieck group:
\[
D[M] = [M^{*}].
\]
\end{thm}

\subsubsection{The relative version of involution via comparison: the finite case and the affine case}

The second part of this article is our attempt to generalize the involution
formula of Theorem \ref{thm:KatoEn} to the Hecke algebras attached to an arbitrary
Bernstein block, and we require such generalization compatible with Aubert-Zelevinsky duality restricted to an
arbitrary Bernstein block. Instead of attacking this problem directly, we step back
to the finite group case and start by generalizing Howlett-Lehrer's results for
representations of the finite endomorphism algebras (Hecke algebra). This result can be viewed as the counterpart of Alvis-Curtis-Kawanaka duality restricted to a
Harish-Chandra series. Then we proceed to study the $p$-adic case guided
by what we see from finite case, and we obtain
\refeq{eq:DH} as our result of comparison with Aubert-Zelevinsky duality.

We now consider the finite Weyl groups and their associated Hecke
algebras. Let $G$ be a $\Fq$-split finite reductive group. We fix a Borel
subgroup $B$, denote its maximal torus by $T$. From the theory of $BN$-pairs we
know $(W:=N_G(T)/T, S)$ ($S$ is determined by $B$) is a Coxeter system. For a subset $I_0 \subset S$, let $W_{I_0}$ be the subgroup of a finite Weyl
group $(W,S)$ generated by $s_{\alpha} \in I_0$, the standard parabolic subgroups containing $B$ are $P_I:
=BW_IB$ for $I \subset S$. Let $\Lambda$ denote an irreducible cuspidal representation of
$L_{I_0}$, $\chi_{\Lambda}$ denote its character. We define the ramification
group $W(\Lambda)$ as $ W(\Lambda) := \{ w \in S_{I_{0}} \ | \ \chi_{\Lambda}
\circ w = \chi_{\Lambda} \}$ which is almost a reflection group. We want to
mention that in the finite case, the $G$-conjugacy class of the pair $( L_{I_0},
  \Lambda)$ plays the same role as ``inertial support $\mathfrak{s}$'' in
  $p$-adic case:  $( L_{I_0},
  \Lambda)$  (\emph{resp.} $\mathfrak{s}$)  determines a subset (\emph{resp.} subcategory) of all
  irreducible representations (\emph{resp.} $\Rep (G(F))$). R.B. Howlett and G.I. Lehrer
\cite[Theorem and Corollary 1]{HowlettLehrer1982} proved

  \begin{equation}\label{eq:HowlettLehrerCorEn}
     \sum_{I \subset S} (-1)^{|I|} \sum_{w \in W_I \backslash C_{I_0}(I) / W(\Lambda)} \Ind^{W(\Lambda)}_{W(\Lambda) \cap  W_I^{w}} (\Res^{W(\Lambda)}_{W(\Lambda) \cap  W_I^w} (\chi))  = \widehat{\chi}:= (-1)^{|I_0|} (-1)^{\ell_{I_0^{\perp}}(-)} \chi ,
  \end{equation}
  where we set $C_{I_0}(I) =  \{ w
  \in W \ | \ w I_0 \subset I \}$  for any $I \subset S$ and
  $\ell_{I_0^{\bot}}$ is the length function of $W(\Lambda)$ associated with
  some Coxeter system (see Lemma \ref{lem:W=CR} and Definition \ref{defn:lIbotpw}) in $<I_0>^{\bot}$.

If we take $I_0= \emptyset$ and replace $W(\Lambda)$ by $W$, we see that \refeq{eq:HowlettLehrerCorEn}
degenerates to \refeq{eq:SolomonthmEn}.

By the comparison of  $\Irr_{\C} ( W( \Lambda))$, Harish-Chandra series $\Irr_{\C} (G|(L_{I_0},
\Lambda))$ and the simple modules of the Hecke algebra $E_G(\Lambda):
= \End_{G}(\Ind^{G}_{P_{I_0}} \Lambda)$ in Section \ref{sec:ResDualityFinite},
we have our second main theorem:

\begin{thm}\label{thm:HLanalogue0En}
 Assume that $W(\Lambda)$ is truly a reflection group\footnote{See Assumption \ref{assmp:CLambdaTrivial}.}. Let $M^{*}$ denote the module $M$ endowed with the twisted action of $E_G(\Lambda)$
 defined by\footnote{In the introduction, we assume equal parameters to simplify our notation, while the general case is stated and proved in the body of this article.}
 \[
   T_w^{*}=(-1)^{|I_0|} (-q)^{\ell_{I_0^{\perp}}(w)}T_{w^{-1}}^{-1}.
 \]
 Then we have the following equality in the Grothendieck group of $E_G(\Lambda)$-modules:
  \begin{equation}
    \label{eq:thm:HLanalogue0En}
    \sum_{I \subset S} (-1)^{|I|} \sum_{w \in W_I \backslash C_{I_0}(I) / W(\Lambda)} [\Ind_{E'_I}^{E_G(\Lambda)}[\Res_{E'_I}^{E_G(\Lambda)}(M)]] = [M^{*}],
  \end{equation}
where  $E_G(\Lambda)$ is defined as $\End_G(\Ind_{P_{I_0}}^G (\Lambda))$ with $\Lambda$ viewed as a $P_{I_0}$-module
  through the natural lifting $P_{I_0} \rightarrow L_{I_0}$ and $E'_I$ is
  the subalgebra of $E_G(\Lambda)$ spanned by $\{ T_{v} \ | \ v \in W(\Lambda) \cap
  W_I^w \}$.
\end{thm}
If we take $I_0 = \emptyset$, then $C_{I_0}(I) = W$ and $P_{I_0} = P_{\emptyset}
= P_0$ is the minimal parabolic subgroup with $L_{I_0} = T$. For $\Lambda$
equals the trivial character of $T$, we have $W(\Lambda) = W$, while $E_G(1) = \End_G
(\Ind_{P_0}^G (1)) $ is just the finite Hecke algebra $H$ with respect to
$(W,S)$, and $E'_I$ is the subalgebra of $H$ spanned by $\{ T_{v} \ | \  v
\in W_I \}=H_I$. Hence, we see that \refeq{eq:thm:HLanalogue0En} provides in
this setting a version of \refeq{thm:KatoEn} for finite Hecke algebra.

Motivated by the results for finite Hecke algebras, we follow the same approach to study the affine case. The Section \ref{subsec:comparison}  is devoted to comparison of the
Aubert-Zelevinsky duality on group side and the involution on the Hecke algebra
side. Following \cite{Roche2002}, we deduce two diagrams \refeq{eq:RocheRUop1}
and \refeq{eq:RocheInd1} that give the  courterparts of the
normalized induction and of Jacquet modules on the categories of right modules
over some endomorphism algebra. We can thus study the corresponding involution
(see \refeq{eq:DH}) and show the relation with the  two previous
involutions in Theorem \ref{thm:KatoEn} and \ref{thm:HLanalogue0En}.

\subsection*{Acknowledgments.} The author would like to thank Anne-Marie Aubert  for suggesting this subject and for invaluable guidance throughout
the preparation of his PhD thesis.  The author also own a lot to Ramla Abdellatif, Noriyuki Abe and David Renard for suggestions on modifications. Besides, the author wants to thank  Peiyi Cui, Long Liu, Javier Navarro, Maarten Solleveld and Zhixiang Wu for mathematical discussions on related subjects. Finally, the author gratefully acknowledges the research environment provided by Sorbonne University and the ITS at Westlake University, as well as research funding support from Professor Huayi Chen.

\section{General setup}\label{sec:Setup}

\subsection{Notation and conventions}\label{sec:Notation} We begin by introducing the notations and conventions that will be used throughout this article.\\ $\operatorname{Card} (S)$: the number of elements in a set $S$.\\
$F$: a non-archimedean local field (starting in Section \ref{sec:Reppadic}).\\
$\mathfrak{o}$: ring of integer of $F$.\\
$\varpi_F$: a fixed prime element of $F$. \\
$\nu_F$: normalized absolute value of $F$.\\
$\mathfrak{p}$: prime ideal in $\mathfrak{o}$, residue field $\Fq = \mathfrak{o} /
\mathfrak{p}$, $q = \operatorname{Card} (\mathfrak{o} /
\mathfrak{p})$.\\
Let $R$ be a ring with identity,  $R \textrm{-}\Mod$ ($\Mod \textrm{-} R$) denote the category of
left (\emph{resp.} right) unital $R$-modules.

  \subsection{Root systems and affine Weyl groups}\label{subsec:RootSystemAffineWeyl}
  A quadruple $\R:= (X, R,Y,R^{\vee})$ is called \emph{root datum} if the following
conditions are met:
\begin{itemize}
\item $X$ and $Y$ are lattices of finite rank, with a perfect pairing $\left< \cdot ,
    \cdot \right> \rightarrow \Z$.
\item $R$ is the root system in $X$.
\item  $R^{\vee}$ is the coroot (dual root) system such that $\left< \alpha,
    \alpha^{\vee} \right>=2$.
  \item For every $\alpha \in R$, $x \in X$, $s_{\alpha}(x):= x -\left<x, \alpha^{\vee}
    \right>\alpha$ acts on affine space $E := X \otimes \mathbb{R}$ and stabilizes $R$.
 \item  For every $\alpha^{\vee} \in R^{\vee}$, $u \in Y$, $s^{\vee}_{\alpha}(u):= u
-\left< \alpha, u \right>\alpha^{\vee}$ acts  on affine space $E^{\vee}:= Y \otimes \mathbb{R}$ and stabilizes $R^{\vee}$.
\end{itemize}

Let $\Delta$ be a set of simple roots, and $R^{+}$ the positive roots with respect to $\Delta$.

Let $W=W(R)$ denote the group generated by $s_{\alpha}$ for $\alpha \in R$, and $S=\{ s_{\alpha}: \alpha \in
\Delta  \}$, $(W,S)$ is a finite Coxeter system. Let us denote the longest
element of $W$ by $w_0$. $W$ can be identified with
a subgroup of $\GL(X \otimes \mathbb{R})$ generated by $\{ s_{\alpha} \ | \ \alpha
\in  S \}$ or a subgroup of $\GL (Y \otimes \mathbb{R})$ generated by $\{ s^{\vee}_{\alpha} \ | \ \alpha
\in  S \}$, and the pairing $\left< \cdot, \cdot \right>$ extends naturally to $X \otimes \mathbb{R}$ and $Y
\otimes \mathbb{R}$.  Let $W(\R):= W \ltimes X$ be
the extended affine Weyl group, it contains a normal subgroup $\Waff$, which is defined as
$W \ltimes \Z R $.  We will mainly consider the (extended) affine Weyl group
acting on $X \otimes \mathbb{R}$. Let $R_{\operatorname{max}}^{\vee}$ be the set of maximal elements of
$R^{\vee}$, with respect to the dual base $\Delta^{\vee}$. For simplicity, in
the proofs we will assume $R$ is irreducible. Under this assumption, let
$\alpha_0^{\vee}$ be the unique maximal coroot and $s_0 = s_{\alpha_0} t_{-\alpha_0}$,
where $s_{\alpha}$ $(\alpha \in R)$ is the reflection with respect to $\alpha$
and $t_{x}$ $(x \in X)$ is translation by $x$. Define $\Saff: =
S \cup \{ s_0  \}$, we know that $(\Waff, \Saff)$ is a Coxeter system.

The $\WR$ can be seen as acting on the affine space $X \otimes \mathbb{R}$ by translation and reflections
extending the action on $X$. The hyperplanes $H_{\alpha,n}:= \{ x \in X \otimes
\mathbb{R} \ | \ \left<x, \alpha^{\vee} \right> =n \}$ are $\WR$ stable and
divides $X \otimes \mathbb{R}$ into alcoves. The choice of $\Delta$
determines a fundamental alcove  $A_0$ in $X \otimes \mathbb{R}$, the unique
alcove contained in the positive Weyl chamber (with respect to $\Delta$), such
that $0 \in \overline{A}_0$. Put $\Omega$ as the
stablizer of $A_0$ in $\WR$. We have the splitting $\WR = \Omega \ltimes
\Waff$. We have a formula of the length function for $w :=\wfin
 t_x \in \WR$:

The length function for $\WR$, $\ell_{\WR}: \WR= W \ltimes X \rightarrow \Z_{\geq 0}$ is given by
\begin{equation}
  \label{eq:Length}
  \ell_{\WR} (\wfin
  t_x)= \sum_{\alpha \in R^{+} \cap \wfin^{-1}(R^{+})} | \left<
  x, \alpha^{\vee}  \right>| + \sum_{\alpha \in R^{+} \cap \wfin^{-1}(-R^{+})} |1+
\left<x,  \alpha^{\vee} \right>| .
\end{equation}
When no confusion is caused, we usually omit the sub-index of $\ell_{\WR}$.

\subsection{Affine Hecke algebras}
Let $K$ be an arbitrary field of characteristic $0$, $(\mathbf{W},\mathbf{S})$ be a Coxeter system
($(W, S)$ or $(\Waff, \Saff)$). Let $\mathbf{q}= (q_s)_{s \in \mathbf{S}} $ be
set of indeterminates satisfying $q_s = q_t$ whenever $s$ and $t$ are conjugate.
\begin{defn}[Generic Hecke algebra]\label{defn:GenericHecke}
  The generic Hecke algebra $\H_{\mathbf{q}}(\mathbf{W},\mathbf{S})$ is the $\Z [\mathbf{q},
  \mathbf{q}^{-1}]$-algebra generated by $ \{T_s\}_{s \in \mathbf{S}}$ satisfying
   \begin{equation}
    \label{eq:GenericHecke1}
    (T_s+1)(T_s -q_s) =0 \textrm{, for } s \in \mathbf{S} \textrm{ (quadratic relation)},
  \end{equation}
  \begin{equation}
    \label{eq:GenericHecke1}
   T_{s}T_t T_s \cdots = T_tT_sT_t \cdots \textrm{for all }sts \ldots = tst \ldots \in \mathbf{W} \textrm{ (braid relations)}.
  \end{equation}
\end{defn}

We now extend the definition of indeterminates and specialize $q_s$. Let $\lambda$, $\lambda^{*}: R \rightarrow K $ be functions such that
\begin{enumerate}[(1)]
\item if $\alpha$, $\beta \in R$ are $W$-conjugate, then $\lambda (\alpha) =
  \lambda (\beta)$ and $\lambda^{*}(\alpha) = \lambda^{*}(\beta)$,
\item if $\alpha^{\vee} \not\in 2 Y$, then $\lambda^{*}(\alpha) = \lambda (\alpha)$.
\end{enumerate}
For $\alpha \in R$, we require $q_{s_{\alpha}} = q^{\lambda(\alpha)}$ and if
$\alpha^{\vee} \in R_{\operatorname{max}}^{\vee}$, we write $q_{s'_{\alpha}}=
q^{\lambda^{*}(\alpha)}$ where $s'_{\alpha} = s_{\alpha}t_{-\alpha}$. For $w = \gamma \tau \in \Omega \ltimes \Waff$, we
define $q(w)$ as the function $q(w) = q(\tau) = \displaystyle \prod_{s \in \Saff} q_s^{n_s} $
with $n_s $ equals the multiplicity of $s$ in a reduced decomposition of $w$ with
respect to $\Saff$.

\begin{defn}[Finite Hecke algebra]
  \label{defn:FiniteHeckeAlgebra}
   For the finite Coxeter system $(W,S)$ and parameters $q_s \in K^{\times} \backslash \{ \textrm{roots of unity}\}$, define the
   finite Hecke algebra $H= H (W, q_s)$ as
  the $K$-algebra with basis $T_w$, $w \in W$ satisfying:
  \begin{equation}
    \label{eq:HeckeAlgDef0}
    (T_s+1)(T_s -q_s) =0 \textrm{, for } s \in S,
  \end{equation}
  \begin{equation}
    \label{eq:HeckeAlgDef}
    T_w \cdot T_{w'} =T_{ww'} \textrm{, if }\ell(w) + \ell(w') = \ell(ww').
  \end{equation}
\end{defn}

\begin{defn}[Iwahori-Matsumoto]\label{defn:HeckeIM}
  For $q_s \in K^{\times} \backslash \{ \textrm{roots of unity}\}$, define the
  extended affine Hecke algebra $\H= \H (\WR, q_s)$ as
  the $K$-algebra with basis $T_w$, $w \in \WR$ satisfying:
  \begin{equation}
    \label{eq:HeckeAlgDef0}
    (T_s+1)(T_s -q_s) =0 \textrm{, for } s \in \Saff ,
  \end{equation}
  \begin{equation}
    \label{eq:HeckeAlgDef}
    T_w \cdot T_{w'} =T_{ww'} \textrm{, if }\ell(w) + \ell(w') = \ell(ww').
  \end{equation}
\end{defn}
Now we begin to introduce the Bernstein-Lusztig presentations. Let $\{ \theta_{x} \ | \ x \in X \}$ be the standard basis of $\C [X]$.
\begin{defn}[Bernstein-Lusztig]
  \label{defn:HeckeAlgBernsteinLusztig}
 The algebra $\H (\R, \lambda, \lambda^{*}, q)$ is the vector space $K [X]
 \otimes_{K} \H (W, q)$ with the multiplication rules:
 \begin{enumerate}[(1)]
 \item $K[X]$ and $H (W, q)$ are embedded as subalgebras,
 \item for $\alpha \in \Delta$ and $x \in X$:
   \[
\theta_{x} T_{s_{\alpha}} - T_{s_{\alpha}} \theta_{s_{\alpha}(x)} =
((q^{\lambda(\alpha)} -1) + \theta_{-\alpha} (q^{(\lambda (\alpha) +
  \lambda^{*}(\alpha))/2} -q^{(\lambda (\alpha) - \lambda^{*}(\alpha))/2}))
\frac{\theta_{x} - \theta_{s_{\alpha}(x)}}{\theta_{0}-\theta_{-2 \alpha}}.
   \]
 \end{enumerate}
\end{defn}

Denote the Iwahori-Hecke algebra associated with $(\Z R, R, \Hom_{\Z} (\Z R, \Z),
R^{\vee}, \Delta)$ by $\H (\Waff, q_s)$ with $q_s \in \C$. There is a group
action of $\Omega$ on $\H (\Waff, q_s)$ given by:
\[
 T_{\omega w \omega^{-1}} = \omega \cdot T_w \cdot \omega^{-1}
\]
The crossed product algebra $\H (\Waff, q_s) \rtimes \Omega$ is the vector space
$\H (\Waff, q_s) \otimes K [\Omega ]$ equipped with the above group action and
the Hecke algebra structure of $\H (\Waff, q_s)$. We have
\[
 \Omega   \ltimes \H (\Waff, q_s) \cong  \H (\WR, q_s).
\]
See \cite[Section 3]{Lusztig1989} for more details. From \cite{Solleveld2021} and \cite[Section~3]{Lusztig1989}, we have the
following:
\begin{thm}[Bernstein]\label{thm:Iwahori=Bernstein}
Pick $\lambda (\alpha), \lambda^{*} (\alpha)$ such that $q_{s_{\alpha}} =
q^{\lambda (\alpha)}$ for all $\alpha \in R$ and $q_{s'_{\alpha}} \in
q^{\lambda^{*} (\alpha)}$ when $\alpha^{\vee} \in
R_{\operatorname{max}}^{\vee}$. Then there exists a unique algebra isomorphism
$ \H (\Waff, q_s) \rtimes \Omega \rightarrow \H (\R, \lambda, \lambda^{*}, q_s) $ such that:
\begin{enumerate}[(1)]
\item it is the identity on $\H (W,q)$,
 \item for $x \in \Z R^{\vee} \cap X_{\operatorname{dom}}$, its sends $q(t_x)^{-1/2}T_{t_x}$
   to $\theta_x$.
\end{enumerate}
\end{thm}

For any $\mu \in X$ denote for simplicity that $T_{\mu} =
T_{t_{\mu}}$. Let us define $\TT_{x}$, $x \in X$ as in Section 2.6 of \cite{Lusztig1989}:
\begin{equation}
  \label{eq:def:TT}
  \TT_{x}:=T_{x+\mu}T_{\mu}^{-1} ,
\end{equation}
where $ \mu \in X_{\operatorname{dom}}$ is chosen such that $x+\mu $ and
$\mu$ are both dominant. It is easy to verify that $\TT_{x}$ does not depend on the choice of $\mu$.
We see $\TT_x = q^{\frac{1}{2} \widetilde{L}(t_{x})}\theta_{x}$, where $
\widetilde{L}(t_{x})$ is a function defined by Lusztig in  \cite[Section 3.1]{Lusztig1989}. Define
$\TT_w = T_{\wfin} \TT_{x}$ for $w = \wfin t_{x} \in W(\R)$, we know from
\ref{defn:HeckeAlgBernsteinLusztig} that $ \{ \TT_w \ | \ w \in \WR \}$  form a basis of $\H$.

 Recall our notation for the finite Hecke algebra $H= H (W, q_s)$ and the
 affine Hecke algebra  $\H= \H (\WR, q_s)$. We now introduce more notions.  For a subset $I \subset S$, let $W_I$ denote the subgroup of $W$ generated by
$I$. We define $H_{I}$ as the subalgebra of the finite Hecke algebra generated
by $T_w$ for $ w \in W_I$. Let us define $\H_{\emptyset} = \sum_{\mu \in X} K \cdot
\TT_{\mu}$ and
$\H_I = H_I \otimes \H_{\emptyset}$ which is isomorphic to the affine Hecke
algebra associated with $\WR_I =
W_I \ltimes X$.
Following \cite[Theorem 3.3]{IwahoriMatsumoto1965}, we have the following
decomposition theorem for $T_w$ with $w \in \WR$:
\begin{thm}\label{thm:TwDecomp}
The algebra  $\H$ in Definition \ref{defn:HeckeIM} is generated by $T_{\rho}$, for $ \rho \in \Omega$ and $T_{s}$, for $s \in
  \Saff $. We write an element $w$ of the affine Weyl group $\WR$ as
  $w := \gamma \tau$, $\gamma =w_{\Omega} t_{\mu} \in \Omega$ with $w_{\Omega}
  \in W$, $\mu \in X$, $\tau \in \Waff$, and $
  \tau = s_{i_1} \cdots s_{i_r}$, $s_{i_k} \in \Saff $ be a reduced expression
  of $\tau$. Then
  \[
T_w = T_{\gamma} T_{\tau} = T_{\gamma} T_{s_{i_1}} T_{s_{{i_2}}} \cdots
T_{s_{i_r}} .
  \]
\end{thm}
The restriction of an $\H$-module $M$ to $\H_I$ is just the
  restriction in the usual sense, denoted by
  $\Res_IM$.

  Now we introduce the induced module $(\pi= \Ind_I \sigma, \Ind_IN)$ for an
  $\H_I$-module $(\sigma, N)$.

  \begin{defn}\label{defn:indHecke}
    The induced module $(\pi= \Ind_I \sigma, \Ind_IN)$ (often appears as $\H
    \otimes_{\H_I} N$) is the module with space $\left( \H
      \otimes_{\H_{\emptyset}} N \right) / K(\sigma)$ where $K(\sigma)$ is the subspace
    generated by
    \[
      \left< hT_s \otimes n - h \otimes \sigma
        (T_s) n ;  \ h \in \H, \ s \in I, \ n \in N \right>,
    \]
    and the Hecke algebra action given by

    \[
      \pi (h_1) [h_2 \otimes n  + K(\sigma)] =
      h_1h_2 \otimes n + K(\sigma), \ h_1, h_2 \in \H.
      \]
  \end{defn}

\section{The Involution theorem d'après S-I. Kato}\label{chp:involution}

We use the notations from previous Section. Recall that $H=H(W,q_s)$ is the
finite Hecke algebra associated with the Coxeter system $(W,S)$, and $\H = \H (\WR , q_s)$ is the generalized
affine Hecke algebra defined in Definition \ref{defn:HeckeIM} with not
necessarily equal parameters. We make the assumption that the root system $R$ is
irreducible. Let $\mathfrak{R}(\H)$ be the Grothendieck group of finite dimensional
 $\H$-modules over $K$. We define the dual of an $\H$-module $ (\pi, M)$ (we
 omit $\pi$ when no confusion can be caused) by
 \begin{equation}
   \label{eq:InvolutionH}
   D[M] = \sum_{I \subset S} (-1)^{|I|}[\Ind_I (\Res_I M)].
 \end{equation}
Define

\begin{equation}
  \label{eq:*directdef}
  T_w^{*} := (-1)^{\ell(w_{\Omega})} \prod_{k=1}^r (-q_{s_{i_k}})T_{w^{-1}}^{-1} ,
\end{equation}
where notations follow Theorem \ref{thm:TwDecomp} in the decomposition
of $T_w$. For an $\H$-module $(\pi, M)$,
denote $(\pi^{*}, M^{*})$ the $\H$-module obtained by twisting the action of $\H$ by $*$.
To be more precise, the module $M^{*}=M$ as $K$-vector space, equipped with the
 twisted action $\pi^{*} (h) (m) :=\pi (h^{*}) (m)$ for $\forall h \in \H$ and $\forall m
\in M$.  We define
\[
  q(w) =  \prod_{k=1}^r q_{s_{i_k}}.
\]
We recall notations here:  $w
=\wfin t_{x}$ is the decomposition for $w$ as an element in $ W \ltimes X $; and $w =  \gamma
\tau$ is the decomposition for $w$ as an element in $ \Omega \ltimes
\Waff$, with $\gamma \in \Omega$ with length $0$ and $\tau$ a product of $\ell(w)$
elements in $\Saff$. Moreover, $\gamma \in \Omega$ can further be written as $w_{\Omega} t_{\mu}$, $w_{\Omega} \in
W$, $t_{\mu} \in X$ and $\tau = w_{\Omega} t_{\mu} s_{i_1} s_{i_2} \cdots s_{i_r}$ where
$r = \ell(w) =\ell(\tau)$. The above defined involution
on Hecke algebras can be rewritten as
\begin{equation}
  \label{eq:T*rough}
  T_w^{*} =
(-1)^{\ell(w_{\Omega})+\ell(w)} q(w) T_{w^{-1}}^{-1}.
\end{equation}

\begin{prop}
  \label{prop:lengthmod2}
  We have $T_w^{*} = (-1)^{\ell(\wfin)} q(w) T_{w^{-1}}^{-1}$ for
  \[
    w =\wfin t_{x} = \gamma
    \tau = w_{\Omega} t_{\mu} s_{i_1} s_{i_2} \cdots s_{i_r}.
    \]
\end{prop}

\begin{proof}
  We need to show that $\ell(w_{\Omega}) + \ell(w)  \equiv \ell( \wfin) \mod 2$. Notice that $w = w_{\Omega}
  t_{\mu} \tau  = \wfin t_{x}$. Among all simple reflections appearing in the
  decomposition $
 \tau= s_{i_1} s_{i_2} \cdots s_{i_r}$, some of them are $s_0 =
  s_{\alpha_0}t_{-\alpha_0}$, assume that there are  $N$ such places.
  To illustrate the idea, we specify one such term with largest index ($k$-th
  place) by writing
  \[
    \tau =  s_{i_1} s_{i_2} \cdots
  s_{i_r} = s_{i_1} s_{i_2} \cdots s_{i_{k-1}} (s_{\alpha_0} t_{-\alpha_0})
  s_{i_{k+1}}
  \cdots s_{i_r}.
\]
We also decompose $w_{\Omega}$ into a product of $\ell(w_{\Omega})=n$ elements: $w_{\Omega}
  = s_{j_1} s_{j_2}\cdots s_{j_n}$. Put these into the expression of $w$, we get
  \begin{equation}
    \label{eq:wUgly1}
    w = s_{j_1} s_{j_2}\cdots s_{j_n} t_{\mu} s_{i_1} s_{i_2} \cdots s_{i_{k-1}} (s_{\alpha_0} t_{-\alpha_0})
  s_{i_{k+1}}
  \cdots s_{i_r} .
  \end{equation}
 Notice that for any element $w_0 \in W$, $\lambda_0 \in
  X$,  we always have the relation
  $w_0t_{\lambda_0}w_0^{-1}=t_{w_0(\lambda_0)}$. Thus $t_{- \alpha_0}  s_{i_{k+1}}
  \cdots s_{i_r} = s_{i_{k+1}}
  \cdots s_{i_r} t_{-( s_{i_{k+1}}
    \cdots s_{i_r})^{-1} \alpha_0}$. Put this into the expression for $w$, we get
\begin{equation}
    \label{eq:wUgly2}
   w= w't_{( s_{i_k}
  \cdots s_{i_r})^{-1} \alpha_0}, \quad w'= s_{j_1} s_{j_2}\cdots s_{j_n} t_{\mu} s_{i_1} s_{i_2} \cdots s_{i_{k-1}} s_{\alpha_0}  s_{i_{k+1}} \cdots s_{i_r} .
  \end{equation}
Repeat this process in $w'$ for the all places where $s_0$ appear to put
translations to the right side, starting from large subindexes to
small subindexes and do it also for $t_{\mu}$. We finally write $w$ as an element in the
finite Weyl group
mulitiplying a translation:
\begin{equation}
    \label{eq:wUgly3}
   w = s_{j_1} s_{j_2}\cdots s_{j_n}  s_{i_1} s_{i_2} \cdots s_{i_{k-1}} s_{\alpha_0}  s_{i_{k+1}}
  \cdots s_{i_r} t_{\phi},\  \phi \in X.
  \end{equation}

  By the uniqueness of the expression for $w = \wfin t_{x}$, we have
 \begin{equation}
    \label{eq:wUgly4}
\wfin = s_{j_1} s_{j_2}\cdots s_{j_n}  s_{i_1} s_{i_2} \cdots s_{i_{k-1}} s_{\alpha_0}  s_{i_{k+1}}
  \cdots s_{i_r},
\end{equation}
as $n+r$ elements product. We will need a few facts:
\begin{enumerate}[(1)]
\item The length $ \ell (s_{\alpha_0}) $ is odd integer because it is
  a reflection.
\item  By the \emph{deletion condition} for general Coxeter system (see \cite[Section 5.8
  Corollary]{Humphreys1990}), we know in obtaining the reduced expression from a not
 necessary reduced one, simple
 reflections are cancelled by pairs. Thus for
 any element of the form $\prod_i u_i$ in a general Coxeter system, $\ell(\prod_i u_i) \equiv \sum_i \ell(u_i)
 \mod 2$.
 \item  Recall that $\ell(w_{\Omega}) = n$, $\ell(w) = \ell(\tau) = r$.
\end{enumerate}
We have \begin{equation}
          \begin{aligned}
        \ell(\wfin)  \equiv & \ell(w_{\Omega}) + \sum_{j=0}^{N} \ell(s_{i_{j,1}} \cdots s_{i_{j,k_{j}-1}}) + N
            \ell(s_{\alpha_0}) \\
            \equiv  & n+  (\sum_{j=0}^{N} \ell(s_{i_{j,1}} \cdots s_{i_{j,k_{j}-1}}) + N
   ) \\
      \equiv     &   n+ \ell( s_{i_1} s_{i_2} \cdots s_{i_{k-1}} s_0
                   s_{i_{k+1}} \cdots s_{i_r}) \\
            = & n +r = \ell(w_{\Omega})+ \ell(w) \mod 2,
          \end{aligned}
        \end{equation}
where $k_j$, $j=1, \ldots, N$ are the places for $s_{\alpha_0}$ to appear in \refeq{eq:wUgly4}, $s_{i_{j,1}} \cdots s_{i_{j,k_{j}-1}}$ is the product of simple
reflections in $W$ between the $j$-th and $j+1$-th $s_{\alpha_0}$.

\end{proof}

The main goal of this Section is to prove the following:
\begin{thm}\label{thm:Kato.main}
  We have $D[M] = [M^{*}]$ for any $\H$-module $ M$, where $*$ is the involution
  on the elements of $\H$ defined by
  \[
    T_w^{*} = (-1)^{\ell(\wfin)} q(w) T_{w^{-1}}^{-1}.
    \]
\end{thm}

\subsection{Proof of the main Theorem}\label{subsec:ProofMain1}

Now we start the proof of Theorem \ref{thm:Kato.main} by writing
\[
\Ind_I (\Res_I M) = \H \otimes_{\H_{\emptyset}} M / \left< hT_s \otimes m - h \otimes \pi
  (T_s) m ;  \ h \in \H, \ s \in I, \ m \in M \right> .
\]

From the Bernstein-Lusztig presentation $\H = H \otimes_{K} \H_{\emptyset}$, we have $\H
\otimes_{\H_{\emptyset}} M \cong H \otimes_{K} M$ as vector spaces.
Define $\tau_s \in \End_{K}(\H \otimes_{K}M)$ for $s \in S$ by
\[
 \tau_s   (h \otimes m) = hT_s \otimes \pi (T_s)^{-1} m - h \otimes m .
\]
 There is a corresponding element of $\tau_s$ in $\End_{K}(\H
 \otimes_{\H_{\emptyset}} M)$, by abuse of notation we still denote it by $\tau_s$.
We have
\[
\Ind_I (\Res_I M) \cong \H \otimes_{\H_{\emptyset}} M / \sum_{s \in I} L_s .\quad
\textrm{(} L_s = \Im \tau_s \textrm{)}
\]
Let us identify $M^{W} = K^{W} \otimes_{K} M$ with $\H \otimes_{K} M$ by the
$K$-linear map:
\begin{equation}
  \label{eq:phi}
\varphi: M^W \rightarrow \H \otimes_{K} M \quad \textrm{defined by }\varphi ((m_w)_{w \in W}) = \sum_{w \in W}T_w \otimes \pi (T_w)^{-1} m_w .
\end{equation}
Composed with the isomorphism $ H \otimes_{K} M \cong \H
\otimes_{\H_{\emptyset}} M $, we have a corresponding map $M^W \rightarrow \H
\otimes_{\H_{\emptyset}} M$ and still denote it by $\varphi$.

\begin{lem}\label{lem:Kato.Lem1}
  For $s \in S$, put
  \[
K_s^{W} = \{ (x_{w})_{w \in W} \in K^{W} \ | \ x_{ws} = - x_w, \ w \in W \}
  \]
Then we have $\varphi^{-1} (L_s) = K_s^W \otimes_K M$.
\end{lem}

\begin{proof}
  Here the computation is done for $H \otimes_K M$, when we apply it we will
  precompose with $H \otimes_K M \cong \H \otimes_{\H_{\emptyset}}M $.
Let $v \in W $ be an element of $W$ with $\ell (vs) > \ell (v)$. Then for $m_v$, $m_{vs}
\in M$, we have
  \begin{equation*}
  \begin{aligned}
    \tau_s(T_v \otimes \pi (T_v)^{-1} m_v + T_{vs} \otimes \pi (T_{vs})^{-1}m_{vs}) &= T_{vs} \otimes \pi (T_{vs})^{-1} m_v -T_v \otimes \pi (T_v)^{-1} m_v\\
    &  + T_{vs} T_s \otimes \pi (T_{vs} T_s)^{-1} m_{vs} - T_{vs} \otimes \pi (T_{vs})^{-1} m_{vs} \\
    & =  A+B+C+D .
  \end{aligned}
\end{equation*}
where $A:= +T_{vs} \otimes \pi (T_{vs})^{-1} m_v  $, $B:= -T_v \otimes \pi
(T_v)^{-1} m_v$, $C:=  + T_{vs} T_s \otimes \pi (T_{vs} T_s)^{-1} m_{vs}$, and
$D:= - T_{vs} \otimes \pi (T_{vs})^{-1} m_{vs}$.

For $s \in S$, we will need the following easy facts:
\begin{equation}\label{eq:HeckeAlgDef1}
 T_s^2 = q_s + (q_s-1)T_s,
\end{equation}
\begin{equation}
  \label{eq:HeckeAlgDef2}
  T_s =\dfrac{q_{s}-T_s^2}{1-q_s},
\end{equation}

\begin{equation}
  \label{eq:HeckeAlgDef3}
  T_s^{-1} =q_s^{-1}T_s - (1-q_s^{-1}).
\end{equation}
We have
\begin{equation*}
  \begin{aligned}
C+D & =  T_v T_s^2 \otimes \pi (T_v T_s^2)^{-1}m_{vs} - T_{vs} \otimes \pi (T_{vs})^{-1} m_{vs} \\
&   = T_v (q_s + (q_s -1 )T_s) \otimes \pi (T_v T_s^2)^{-1} m_{vs} -T_{vs} \otimes \pi (T_{vs})^{-1} m_{vs} \textrm{ (Using }\refeq{eq:HeckeAlgDef1} \textrm{)}\\
& = T_v \otimes q_s \pi (T_v)^{-1} \pi (T_v) \pi (T_v T_s^2)^{-1} m_{vs} - T_{vs} \otimes \pi (T_{vs})^{-1} m_{vs}\\
& + T_{vs} \otimes (q_s -1) \pi (T_{vs})^{-1} \pi (T_{vs}) \pi (T_v T_s^2)^{-1} m_{vs}
  \end{aligned}
\end{equation*}
Rewrite the last term:
\begin{equation*}
  \begin{aligned}
    T_{vs} \otimes (q_s -1) \pi (T_{vs})^{-1} &\pi (T_{vs}) \pi (T_v T_s^2)^{-1} m_{vs} \\
    &= T_{vs}
    \otimes \pi (T_{vs})^{-1} (q_s -1) \pi (T_v T_sT_s^{-2} T_v^{-1}) m_{vs}  \\
   \textrm{ (Using }\refeq{eq:HeckeAlgDef2} \textrm{) } &= T_{vs} \otimes \pi (T_{vs})^{-1} \pi (T_v (T_s^2 - q_s)T_s^{-2} T_v^{-1}) m_{vs} \\
    & = T_{vs} \otimes \pi (T_{vs})^{-1} m_{vs} - T_{vs} \otimes \pi (T_{vs})^{-1} q_s \pi (T_v T_s^{-2}T_v^{-1}) m_{vs}.
  \end{aligned}
\end{equation*}
Hence
\begin{equation*}
  C+D =  T_v \otimes  \pi (T_v)^{-1}q_s \pi (T_v T_s^{-2}T_v^{-1}) m_{vs} -T_{vs} \otimes \pi (T_{vs})^{-1} q_s \pi (T_v T_s^{-2}T_v^{-1}) m_{vs}.
\end{equation*}
and
\begin{equation}
  \label{eq:lemma1}
  A+B+C+D= T_{vs} \otimes \pi (T_{vs})^{-1} m' - T_v \otimes \pi (T_v)^{-1} m',
\end{equation}
where $m'=m_v - q_s \pi (T_v T_s^{-2}T_v^{-1})m_{vs}$.
This shows that the inverse image has $-m'$, $m'$ repectively for places indexed
by $v$, $vs$, thus $\varphi^{-1}(L_s)  \subset K_s^W \otimes_K M$. If we let
$m_{vs}=0$, then $\varphi^{-1}(L_s)$ has $-m_v$ and $m_v$ respectively for
places indexed by $v$, $vs$, thus $\varphi^{-1} (L_s) \supset K_s^W \otimes_K M$.
\end{proof}
Thus we have a $K$-isomorphism
\[
(K^W/\sum_{s \in I} K_s^{W}) \otimes_K M \cong \H \otimes_{\H_{\emptyset}} M /
\sum_{s \in I} L_s .
\]
We can further identify $K^W/\sum_{s \in I} K_s^{W}$ with $K[W/W_I]$ by the map
\[
  (x_w)_{w \in W} \mapsto \sum_{w \in W}x_w \cdot w W_I.
\]

Now we will use the complex introduced by Deligne-Lusztig \cite{DeligneLusztig1982}
and S-I. Kato \cite{Kato1993}. First, define

\[
  \pi^{I'}_{I} :\H \otimes_{\H_{\emptyset}} M / \sum_{s \in I} L_s \rightarrow
  \H \otimes_{\H_{\emptyset}} M / \sum_{s \in I'} L_s ,
\]
as the natural projection for $I \subset I' \subset S$
where $I'$ is
obtained by adding one more element from $S \backslash I$. We know that $\pi^{I'}_I$ is an
$\H$-homomorphism.

To give the set of simple roots $S$ an ordering with respect to the Weyl group
action as in \cite[Section 4 ]{Solomon1966} and thus making the complex below \refeq{eq:Katoeq1.2}
well-defined ($d \circ d =0$), we define:

\[
  \epsilon^{I'}_I : \bigwedge^{|I|} (K^{I}) \rightarrow \bigwedge^{|I'|} (K^{I'})
\]
be the natural isomorphism given by $v \rightarrow  v \wedge s$, $s \in I'
\backslash I$. Here $K^X$ is the free $K$-modules for any set $X$, and put $\bigwedge^{|I|}(K^I) = K$ for $I = \emptyset$.

Then we can define the following complex of $\H$-modules:

\begin{equation}
  \label{eq:Katoeq1.2}
  0 \rightarrow C_0 \xrightarrow{d_0} C_1 \xrightarrow{d_{1}}  \ldots  \xrightarrow{d_{i-1}} C_i  \xrightarrow{d_{i}} \ldots  \xrightarrow{d_{|S|-1}} C_{|S|} \rightarrow 0,
\end{equation}
where \[
  C_i = \bigoplus_{|I|=i} (\H \otimes_{\H_{\emptyset}} M / \sum_{s \in I}L_s) \otimes
  \bigwedge^i (K^I),
\]
and $d_i = \pi_I^{I'} \otimes \epsilon^{I'}_I$ with $|I| = i$.

Then using the complex by \cite{Solomon1966} for simplicial decomposition of $(|S|-1)$-dimensional
sphere, or by \cite{DeligneLusztig1982}'s main theorem, we know the complex
\ref{eq:Katoeq1.2} has a nonzero cohomology only at the degree $0$. Thus
\begin{equation}
  \begin{aligned}
D[M]  & = \sum_{i}(-1)^i [C_i] =[\Ker d_0] \\
       &   = [\bigcap_{s \in S}L_s]  =[\varphi (  \bigcap_{s \in S} ( K_s^W \otimes_K
M))].
  \end{aligned}
\end{equation}

Any element of $ \bigcap_{s \in S} ( K_s^W \otimes_K
M)$ is of the form
$\left( (-1)^{\ell(w)}
  x_1 \otimes m \right)_{w \in W}$ (where $x_1 \in K$) for some $m \in M$. Therefore

\begin{equation}\label{eq:phitochi}
  \begin{aligned}
\varphi \left( \left( (-1)^{\ell(w)}
x_1 \otimes m \right)_{w \in W} \right)=&\sum_{w \in W} T_w \otimes (\pi (T_w)^{-1}(-1)^{\ell(w)}
x_1 \otimes m )   \\
  =&    \sum_{w \in W}(-1)^{\ell(w)} T_w \otimes \left( \pi (T_w)^{-1}
    (1 \otimes x_1 m) \right)\\
  =&    \sum_{w \in W}(-1)^{\ell(w)} (T_w \otimes  \pi (T_w)^{-1})
    (1 \otimes (1 \otimes x_1 m))\\
  = & \chi (1 \otimes x_1m) \quad (\textrm{Identify } 1 \otimes x_1 m \textrm{ with } x_1m),
  \end{aligned}
\end{equation}
where $\chi =\sum_{w \in W} (-1)^{\ell(w)}T_w \otimes T_w^{-1}$, acting by
  ``$\textrm{left multiply } T_w \otimes \pi$''.

  We summarize what we have got so far: $D[M] =[\varphi \left(  \bigcap_{s \in S} ( K_s^W \otimes_K
M)\right)]= \chi (1 \otimes_{\H_{\emptyset}} M)$ where $1 \otimes_{\H_{\emptyset}} M \subset \H \otimes_{\H_{\emptyset}} M$.
We know that $\chi$ is bijective $K$-homomorphism when restricted to $1 \otimes_{\H_{\emptyset}}
 M$ since $\varphi$ is bijective on  $\bigcap_{s \in S} ( K_s^W \otimes_K
 M)$. We now need to show that $\chi (1 \otimes M)$ is isomorphic to $M^{*}$, or
 equivalently:
 \[
   (T_{w } \otimes  1) \chi = \chi (1 \otimes T_w^{*}) = \chi (1 \otimes (-1)^{\ell(
     \wfin )} q(w) T_{w^{-1}}^{-1}).
 \]
This is proved in the following steps:
\begin{enumerate}[(1)]
  \item We use the Iwahori-Matsumoto presentation of $T_w$ to write
  \[
    T_w \otimes 1=  (T_{\gamma}
  \otimes 1)( T_{s_{i_1}} \otimes 1)( T_{s_{{i_2}}}\otimes 1) \cdots
  (T_{s_{i_k}} \otimes 1),
\]
where $w = \gamma s_{i_1} s_{i_2}
\cdots s_{i_k}$ with $ \gamma = w_{\Omega} t_{\mu} \in \Omega$ and $s_{i_t} \in
\Saff$ for any $1 \leq t \leq k$.
\item  We prove $T_{s_{i_t}} \otimes 1$, $s_{i_t} \neq s_0$
  intertwines with $\chi$: $(T_{s_{i_t}} \otimes 1) \chi = \chi (1 \otimes
  (-q_{s_{i_t}}T_{s_{i_t}}^{-1}))$ using the first part of Lemma \ref{lem:Kato.Lem2}.
\item
 We prove $T_{s_0} \otimes 1$ intertwines with $\chi$: $(T_{s_0} \otimes 1) \chi = \chi
  (1 \otimes (-q_{s_0}T_{s_0}^{-1}))$ using the second part of Lemma
  \ref{lem:Kato.Lem2}.
\item We prove $T_{\gamma} \otimes 1 $ intertwines with $\chi$: $ (T_{\gamma} \otimes
  1)\chi = (-1)^{\ell(w_{\Omega})} \chi (1 \otimes T_{\gamma})$ using \cite[Lemma 3]{Kato1993}..
\item We conclude that
\begin{equation}
  \begin{aligned}
 (T_w \otimes 1) \chi = &  (T_{\gamma}
  \otimes 1)( T_{s_{i_1}} \otimes 1)( T_{s_{{i_2}}}\otimes 1) \cdots
       (T_{s_{i_k}} \otimes 1) \chi   \\
    = & \chi ((-1)^{l(w_{\Omega})}  (1 \otimes
  T_{\gamma})) (1 \otimes (-q_{s_{i_1}}T_{s_{i_1}}^{-1}))
  (1 \otimes (-q_{s_{i_2}}T_{s_{i_2}}^{-1}))  \cdots (1 \otimes
       (-q_{s_{i_k}}T_{s_{i_k}}^{-1})) \\
    =&   \chi (1 \otimes (-1)^{\ell (w_{\Omega}) +
       l(w)} (\prod q_{s_i})  T^{-1}_{w^{-1}}) \\
    = & \chi (1 \otimes (-1)^{\ell(
    \wfin )} q(w) T_{w^{-1}}^{-1})
  \end{aligned}
\end{equation}
by proposition \ref{prop:lengthmod2}.
\end{enumerate}
From the above proof we can claim:
\begin{prop}
The $*$ operation on $\H$ is an involution (\emph{i.e.} an automorphism of
such that $* \circ * = \operatorname{Id}$).
\end{prop}

\subsection{A Key Lemma}

We recall those definitions introduced in Section
\ref{sec:Setup} and along the way introduce more notions. Let $X$ and $Y$ be lattices of finite rank, with a perfect pairing $\left< \cdot ,
    \cdot \right> \rightarrow \Z$.  Let $R$ be a root system in $X$ and
  $R^{\vee}$ is the coroot (dual root) system such that $\left< \alpha,
    \alpha^{\vee} \right>=2$. For every $\alpha \in R$, $x \in X$, $s_{\alpha}(x):= x -\left<x, \alpha^{\vee}
    \right>\alpha$ acts on affine space $E:=X \otimes \mathbb{R}$ and stabilizes
    $R$. The pairing $\left< \cdot, \cdot \right>$ extends naturally to $X
    \otimes \mathbb{R}$ and $Y \otimes \mathbb{R}$, and it can also be viewed as
    a positive definite inner product on $V$ (the underlying $\mathbb{R}$-vector space of $E$) via the natural map $\iota: V
    \rightarrow V^{\vee}$. The finite Weyl group $W=W(R)$ is the group generated
    by $s_{\alpha}$ for $\alpha \in R$, thus $W$ can be naturally identified with
    a subgroup of $\GL(V)$ acting on $E$ on the left. This is called the \emph{ordinary
    left action}.

Following  \cite[Section 1.6]{Kato1985} and \cite[Section 1.2]{Lusztig1980}, we
introduce a \emph{right} action of $\WR$ on $E$: for $v \in E$, $w \in W$ and $\lambda
\in X$, the extended affine Weyl group $\WR$ acts on $E$ on the right as follows:
\begin{equation}
  \label{eq:rightWact}
  v w := w^{-1} \cdot v,
\end{equation}
where $w^{-1} \cdot v$ is the ordinary left action, and
\begin{equation}
  \label{eq:rightXact}
   v t_{\lambda} :=  v + \lambda .
\end{equation}
 Recall that $\alpha_0^{\vee}$ is the unique maximal coroot and $s_0 = s_{\alpha_0} t_{-\alpha_0}$, a hyperplane indexed by $\alpha \in R^{+}$ and $n \in
 \Z$ is defined as:
 \[
   H_{\alpha,n}= \{ x \in E \ | \ \left<x, \alpha^{\vee} \right> =n \}.
 \]
With respect to the right action, $s_0$ corresponds to the affine reflection
associated with $ H_{\alpha_0,-1}$. We denote the set of all connected
components of $E \backslash (\bigcup_{\alpha \in R^{+}, n \in \Z} H_{\alpha, n})$ by $\mathcal{A}$. We define $A^{-}$ as the open alcove in $E$
bounded by $H_{\alpha, 0}$ for all $\alpha \in R^{+}$ and $H_{\alpha_0, -1}$.
Then $A^{-} \in \mathcal{A}$, and $\mathcal{A} = \{ A^{-}w \ | \ w \in \Waff
\}$. There is a left action of $\Waff $ on $\mathcal{A}$ by $y(A^{-}w)= A^{-}yw$
for $y$, $w \in \Waff$. Let $C^+$ denote \emph{the dominant cone} (\emph{i.e.} the cone
bounded by $H_{\alpha, 0}$ for $\alpha \in R$ and containing  $A^-w_0$).
For any $A \in \mathcal{A}$, we define $ \mathcal{L}(A)$ by
\begin{equation}
  \label{eq:LA}
  \mathcal{L}(A) = \{ s \in \Saff \ | \ A \subset E^{+}_{H_s} \},
\end{equation}
where $H_s$ is the affine hyperplane associated with $s$, and $E^{+}_{H_s}$ is
the complement of $E- H_s$ that has nonempty intersection with $C^+t_{\lambda}$,
for all $\lambda \in X$.

In the rest part of this subsection, we correct some steps in S-I. Kato's proof of Lemma 2 and give more details in Lemma \ref{lem:Kato.Lem2}. The Proof in \cite[Lemma 2]{Kato1993} for $(T_{s_0} \otimes 1) \chi = \chi
  (1 \otimes (-q_{s_0}T_{s_0}^{-1}))$ contains an error when
  he uses the lemma from \cite{Kato1985}: $s \not\in
  \mathcal{L}(A^{-}w)$ is not equivalent to $y^{-1} (\alpha_0) >0$ (using S-I. Kato's
  notation here) in general.
\begin{lem}[S-I. Kato Lemma 2]
  \label{lem:Kato.Lem2}
  We use the definition of a subset $\mathcal{L}(A^-v)$ of $\Saff$ explained as
  above. Then we have
  \[
(T_s \otimes 1) \chi = \chi (1 \otimes (-q_sT_s^{-1}))
  \]
holds in $\H \otimes_{\H_{\emptyset}} \H $ for $s \in \Saff = S \cup \{s_0 \}$. (We
assume the $R$ is irreducible.) Take $v
\in W$ with $s \not\in \mathcal{L}(A^-v)$, we need to show:
  \begin{equation}
    \label{eq:Kato.Lem2.eq1}
    (T_{s} \otimes 1 )(T_v \otimes T_v^{-1} - T_{sv} \otimes T_{sv}^{-1} ) = (T_v \otimes T_v^{-1} - T_{sv} \otimes T_{sv} ^{-1} ) (1 \otimes (-q_s T_s^{-1}) ), \ \textrm{for $s \in S$}
  \end{equation}
  \begin{equation}
   \label{eq:Kato.Lem2.eq2}
    (T_{s_0} \otimes 1 )(T_v \otimes T_v^{-1} - T_{s_{\alpha_0}v} \otimes T_{s_{\alpha_0}v}^{-1} ) = (T_v \otimes T_v^{-1} - T_{s_{\alpha_0}v} \otimes T_{s_{\alpha_0}v} ^{-1} ) (1 \otimes (-q_{s_0} T_{s_0}^{-1}) )
  \end{equation}
\end{lem}
Before coming to the proof of Lemma \ref{lem:Kato.Lem2},  we need the following
two Lemmas from \cite[1.9]{Kato1985} and \cite{IwahoriMatsumoto1965}:
\begin{lem}
  For $s \in \Saff$ and $ v \in W$, we have
  \begin{equation}
    \label{eq:TsTv}
    T_{s} T_v = \left\{
      \begin{aligned}
        \TT_{sv} & , & s \not\in \mathcal{L}(A^-v) \\
        q_{s} \TT_{sv} + (q_{s} -1)T_v & , & s \in \mathcal{L}(A^-v).
      \end{aligned}
    \right.
  \end{equation}
  For $s \in S$, $s \not\in \mathcal{L}(A^-v)$ is equivalent
 to $\ell (sv) = \ell (s) + \ell (v)$ and $\TT_{sv} = T_{sv}$.
\end{lem}

\begin{lem}\label{lem:facts}
  \begin{enumerate}[(1)]
  \item For $s_0 =s_{\alpha_0}t_{-\alpha_0} = t_{\alpha_0}s_{\alpha_0}$, we have
    $T_{\alpha_0}=T_{s_0}T_{s_{\alpha_0}}$.
  \item We have $T_{t_{v(\lambda)}v} =
    T_v T_{\lambda}$ for $\lambda \in X_{\operatorname{dom}}$, $v \in W$.
  \end{enumerate}
\end{lem}
\begin{proof}[Proof of the Lemma \ref{lem:Kato.Lem2}]
  Calculate left hands side of \refeq{eq:Kato.Lem2.eq1}:
  \begin{equation*}
    \begin{aligned}
      (T_{s} \otimes 1 )(T_v \otimes T_v^{-1} - T_{sv} \otimes T_{sv}^{-1} ) &=T_sT_v \otimes T_v^{-1} - T_s T_{sv} \otimes T_{sv}^{-1}  \\
      &   = T_{sv} \otimes T_v^{-1} - (q_sT_v + (q_s -1 )T_{sv} ) \otimes T_{sv}^{-1} \\
      & = -q_s T_v \otimes T_{sv}^{-1} - (q_s -1) T_{sv} \otimes T_{sv}^{-1} + T_{sv} \otimes T_v^{-1}.
    \end{aligned}
  \end{equation*}

Look at the last two terms:

\begin{equation*}
  \begin{aligned}
    - (q_s -1) T_{sv} \otimes T_{sv}^{-1} + T_{sv} \otimes T_v^{-1} & = -(q_s -1) T_{sv} \otimes T_v^{-1} (q_s^{-1} T_s  - (1-q_s^{-1}) ) +T_{sv} \otimes T_v^{-1}  \\
  & = \dfrac{q_s^2 -q_s +1}{q_s} T_{sv} \otimes T_v^{-1} + \dfrac{1-q_{s}}{q_{s}}T_{sv} \otimes T_v^{-1}T_s  .
  \end{aligned}
\end{equation*}

Calculate right hands side of \refeq{eq:Kato.Lem2.eq1}:
\begin{equation*}
  \begin{aligned}
    (T_v \otimes T_v^{-1} - T_{sv} \otimes T_{sv} ^{-1} ) (1 \otimes (-q_s T_s^{-1}) ) & = -q_sT_v \otimes T_v^{-1}T_s^{-1} + q_s T_{sv} \otimes T_{sv}^{-1} T_s^{-1} .
  \end{aligned}
\end{equation*}

Look at the second term above:

\begin{equation*}
  \begin{aligned}
    q_s T_{sv} \otimes T_{sv}^{-1} T_s^{-1} & = q_s T_{sv} \otimes T_v^{-1} (T_s^{-1})^2  \\
   \textrm{  (Using } \refeq{eq:HeckeAlgDef3} \textrm{ )}  &   =q_s T_{sv} \otimes T_v^{-1} (q_s^{-1}T_s - (1-q_s^{-1}))^2 \\
    \textrm{ (Using } \refeq{eq:HeckeAlgDef1} \textrm{ )} & = q_s T_{sv} \otimes T_v^{-1} (q_s^{-2} (q_s + (q_s-1)T_s) + (1-q_s^{-1})^2 -2q_s^{-1} (1-q_s^{-1})T_s) \\
    &  = \dfrac{q_s^2 -q_s +1}{q_s} T_{sv} \otimes T_v^{-1} + \dfrac{1-q_{s}}{q_{s}}T_{sv} \otimes T_v^{-1}T_s .
  \end{aligned}
\end{equation*}
Thus this finishes the verification of \refeq{eq:Kato.Lem2.eq1}. Now check for \refeq{eq:Kato.Lem2.eq2}.

We can assume $s_0 \not\in  \mathcal{L}(A^-v)$, then $s_0 \in
\mathcal{L}(A^-s_{\alpha_0}v)$ and $T_{s_0} T_v = \TT_{s_0v}$ by
\refeq{eq:TsTv}, we have
\[
  \TT_{s_0v} =
\TT_{t_{\alpha_{0}}s_{\alpha_{0}}v} = \TT_{s_{\alpha_{0}}v t_{-v^{-1}
    (\alpha_0)}} = T_{s_{\alpha_0}v} \TT_{-v^{-1}(\alpha_0)}.
\]
On the other
hand, by comparing the length we have $T_{s_0}T_v = T_{s_0 v}$, thus
\begin{equation}
  \begin{aligned}
 T_{s_{\alpha_0}v} \otimes T_{s_{\alpha_0}v}^{-1} = & T_{s_{\alpha_0}v}
\TT_{-v^{-1}(\alpha_0)} \otimes \TT_{-v^{-1}(\alpha_0)}^{-1}
T_{s_{\alpha_0}v}^{-1} \\
    = &   \TT_{s_0v} \otimes \TT_{s_0 v}^{-1} = T_{s_0 v} \otimes T_{s_0
v}^{-1}.
  \end{aligned}
\end{equation}
We see the \refeq{eq:Kato.Lem2.eq2} becomes
  \begin{equation}
    \label{eq:Kato.Lem2.eq2new}
    (T_{s_0} \otimes 1 )(T_v \otimes T_v^{-1} - T_{s_0v} \otimes T_{s_0v}^{-1} ) = (T_v \otimes T_v^{-1} - T_{s_0 v} \otimes T_{s_0 v} ^{-1} ) (1 \otimes (-q_{s_0} T_{s_0}^{-1}) ),
  \end{equation}
whose verification is the same as \refeq{eq:Kato.Lem2.eq1}.
\end{proof}
Following \cite[Lemma 3]{Kato1993}, we have
\begin{lem}[S-I. Kato Lemma 3]\label{lem:Kato.Lem3}
  For $\gamma \in \Omega$ with $\gamma = w_{\Omega} t_{\mu}$, in $\H
  \otimes_{\H_{\emptyset}} \H$
  \begin{equation}
    \label{eq:Kato.Lem3}
    (T_{\gamma} \otimes 1)\chi = (-1)^{\ell(w_{\Omega})} \chi (1 \otimes T_{\gamma}).
  \end{equation}
\end{lem}

\subsection{The involution using the Bernstein-Lusztig presentation}

 We know that $\H(\WR, q_s)$, according to Theorem
  \ref{thm:Iwahori=Bernstein}, has a basis consisting of $\theta_x T_w$ for $x
  \in X$ and $w \in W$. By applying the involution in Theorem
  \ref{thm:Kato.main} to $\theta_x$ in the Bernstein-Lusztig presentation, we can deduce the involution in this setting. This version and
its proof are first seen
in \cite{CiubotaruMasonBrownOkada2022}. We follow their approach, even if we do not
restrict to the equal parameter case, the formula and the proof does not change much.

\begin{prop}\label{prop:InvBLpresentation}
 In the Bernstein-Lusztig presentation, the involution $*$ in Theorem
 \ref{thm:Kato.main} acts as
\[
\theta_x^{*} =T_{w_0} \theta_{w_0 (x)}T_{w_0}^{-1}
\]
and
\[
T_w^{*} = (-1)^{\ell(w)} q(w) T_{w^{-1}}^{-1},
\]
where $w_0$ is the longest element of the finite Weyl group $W$.
\end{prop}
\begin{proof}
  Recall that $\ell (w_0) = \operatorname{Card}(R^+)$, for any $x \in
X_{{\dom}}$ by
\refeq{eq:Length}, we have:
\[
  \ell (w_{0}t_x)= \sum_{\alpha \in R^{+}} (1+   \left<  x, \alpha^{\vee} \right> )
= \ell (w_0) + \ell (t_x) =  \ell (t_{w_0 (x)})+  \ell (w_0).
\]
We have $T_{w_0 (x)} T_{w_0} = T_{w_0} T_x$, thus
\begin{equation}
  \label{eq:BLeq1}
 \theta_x = q(x)^{-1/2} T_x= q(x)^{-1/2}  T_{w_0}^{-1} T_{w_0 (x)} T_{w_0}.
\end{equation}
Apply the involution to both sides of the above equation:
\begin{equation} \label{eq:BLeq2}
  \begin{aligned}
 \theta_x^{*} = & q(x)^{-1/2} q(w_0 (x)) T_{w_0} T_{-w_0 (x)}^{-1} T^{-1}_{w_0} \\
    = & q(x)^{-1/2} q(w_0 (x))^{1/2}  T_{w_0} \theta_{w_0 (x)} T^{-1}_{w_0} \\
     = & T_{w_0} \theta_{w_0 (x)}T_{w_0}^{-1},
  \end{aligned}
\end{equation}
where we used  $q(t_{w_0 (x)})=q(w_0 t_x w_0^{-1})=q(x)$ to get the last equation.
\end{proof}

\subsection{Unitarity of the involution}
\begin{defn}
 Let $A$ be a $\mathbb{C}$-algebra. A star operation on $A$ is a ring
 anti-involution $\kappa: A \rightarrow A$, such that for $a, b \in A$
 \begin{enumerate}
 \item $\kappa(a+b)=\kappa(a)+\kappa(b)$,
 \item $\kappa(a b)=\kappa(b) \kappa(a), a, b \in A$,
 \item $\kappa$ is conjugate-linear, $\kappa(\lambda a)=\bar{\lambda} \kappa(a),
   \lambda \in \mathbb{C}, a \in A$.
 \end{enumerate}
\end{defn}

\begin{defn}\label{defn:star}
 Let $A$ be a $\mathbb{C}$-algebra with a \emph{star operation} $\kappa$ and let $M$ be
 a Hilbert space carries an $A$-module structure. We say that $M$ is $\kappa$-unitary if $M$ has a
 positive-definite inner product $\langle \ ,\ \rangle_M$ which is
 $\kappa$-invariant, \emph{i.e.}:
 \begin{equation}
   \label{eq:unitarity}
   \langle a \cdot m_1, m_2\rangle_M=\langle m_1, \kappa(a) \cdot m_2 \rangle_M,
\quad \textrm{for all } a \in A, m_1, m_2 \in M .
 \end{equation}
\end{defn}
For an affine Hecke algebra $\H$ define over $\C$, there is a natural star operation. In the
Iwahori-Matsumoto presentation it is simply $\kappa (T_w) = T_{w^{-1}}$ for all
$w \in W (\R)$. Let us recall the involution we define earlier $T_w^{*} = (-1)^{\ell(\wfin)} q(w) T_{w^{-1}}^{-1}$ for $w =\wfin t_{x} = \gamma
\tau = w_{\Omega} t_{\mu} s_{i_1} s_{i_2} \cdots s_{i_r}$.

\begin{lem}\label{lem:unitarity}
 Under the above notations, we have $ \langle T^{-1}_{w^{-1}} \cdot m_1,
m_2\rangle_M= \langle  m_1,  T^{-1}_{w} \cdot m_2 \rangle_M$.
\end{lem}
\begin{proof}
  Let $m_1 =T_{w^{-1}} \cdot m_1'$, we have
\begin{equation}
  \begin{aligned}
 \langle m_1, T_w^{-1} \cdot m_2\rangle_M =& \langle T_{w^{-1}} \cdot m_1',
   T_w^{-1} \cdot m_2\rangle_M \\
 =  &   \langle m_1', T_{w} T_w^{-1} \cdot
      m_2\rangle_M=\langle m_1', m_2\rangle_M \\
   = & \langle T_{w^{-1}}^{-1}T_{w^{-1}}
   \cdot m_1',  m_2\rangle_M= \langle T_{w^{-1}}^{-1}m_1,  m_2\rangle_M .
  \end{aligned}
\end{equation}
\end{proof}

\begin{thm}\label{thm:unitarity}
  If the values of $q_s$ take real values, the involution $(\cdot)^{*}$ preserves unitarity.
\end{thm}
\begin{proof}
  From the definition, the underlying space of $M^{*}$ is automatically a Hilbert space and $\langle
  \cdot \ , \ \cdot \rangle_{M^{*}}$ is positive-definite. We need to show  $\langle T_w \cdot m_1, m_2\rangle_{M^{*}}=\langle m_1,
\kappa(T_w) \cdot m_2 \rangle_{M^{*}}$ which is just $\langle T_w^{*} \cdot m_1,
m_2\rangle_M=\langle m_1, \kappa(T_w)^{*} \cdot m_2 \rangle_M$. We have
\begin{equation}
  \begin{aligned}
 \langle
m_1, \kappa(T_w)^{*} \cdot m_2 \rangle_M & = \langle m_1, (-1)^{\ell
                                           ((w^{-1})_{\operatorname{fin}})} q(w^{-1})T_w^{-1} \cdot m_2 \rangle_M \\
    &=\langle m_1, (-1)^{\ell
      (w_{\operatorname{fin}})} q(w)T_w^{-1} \cdot m_2 \rangle_M \\
    &=\langle (-1)^{\ell
  (w_{\operatorname{fin}})} q(w)T_{w^{-1}}^{-1} m_1,  \cdot m_2 \rangle_M= \langle T_w^{*} \cdot m_1,
m_2\rangle_M .
  \end{aligned}
\end{equation}
We have used the fact $\ell (w_{\operatorname{fin}}) = \ell
((w^{-1})_{\operatorname{fin}})$ in the second equation, and used Lemma \ref{lem:unitarity} in the third equation above.
\end{proof}
The Iwahori-Hecke algebra satisfy that $q_s
\in \mathbb{R}$, the involution thus preserves the unitarity. D. Barbasch and A.
Moy in \cite{BarbaschMoy1989} proved the following theorem:
\begin{thm}
  Let $G = \mathbb{G}(F)$ be a split reductive group with connected center over
  $F$, an irreducible module with nonzero Iwahori fixed vector is unitary if and
  only if the corresponding  module of the Iwahori Hecke algebra is unitary.
\end{thm}

\section{A relative version of the involution via the Howlett-Lehrer theory}
\subsection{Notation and conventions}
\noindent $p$: a prime number.\\
$k := \Fq$ with $q = p^n$ for some integer $n \geq 1$.\\
$\bar{k}=\Fpb$: a fixed algebraic closure of $k$.\\
$A$: a commutative ring with unit such that $p$ is invertible in $A$.\\
Let $H$ be a finite group, then we define:\\
$AH \smod$: the category of finite dimensional $AH$ modules.\\
$\mathbf{G}$: a connected reductive algebraic groups
defined over $\bar{k}$, together with an endomorphism $\Fr$, a power of which is
a Frobenius endomorphism. \\
$\Irr_A (\GG^{\Fr})$: the set of simple $A \GG^{\Fr}$-modules up to isomorphism
(also denoted by $\Irr (\GG^{\Fr})$ if no confusion is caused).\\
We keep the notations  for Weyl groups as in section
\ref{subsec:RootSystemAffineWeyl}. We also need the following convention:\\
${}^{g}H:=gHg^{-1}$, and $H^g:=g^{-1}Hg$  where $g$ is an element of group $G$ and $H$ is a set with left-right $G$ action.
For any representation $\pi$ of a group $G$, let $g$ be an element of $G$,
define a representation ${}^{g}\pi$ of ${}^{g}G$ by transport of structure ${}^{g} \pi (x):= \pi (g^{-1}xg)$ for any $x \in {}^{g}G$.

\subsection{A review of the representation theory for finite groups of Lie types}

This subsection is mainly based on \cite[Section 5]{DigneMichel2020} and \cite{DudasMichel2015}.
Recall that $\GG$ is a connected reductive group over $\Fq$. We denote by
$G=\GG^{\Fr}$ the finite group of fixed points.  Let $\mathbf{P}=\mathbf{L} \mathbf{U}$ be a $\Fr$-stable parabolic
  subgroup of $\mathbf{G}$ with $\Fr$-stable Levi complement $\mathbf{L}$. The group algebra  $A
\mathbf{G}^F / \mathbf{U}^F$ over $A$ is viewed as a $(A \mathbf{G}^F ,
A\mathbf{L}^F)$-bimodule where $\mathbf{G}^F $ acts by left translations and
$\mathbf{L}^F$ by right translations. The \emph{Harish-Chandra induction} and \emph{Harish-Chandra restriction} functors are
  $$
  \begin{aligned}
    R_{\mathbf{L} \subset \mathbf{P}}^{\mathbf{G}}: A \mathbf{L}^F\smod & \longrightarrow A \mathbf{G}^F\smod \\
    M & \longmapsto A \mathbf{G}^F / \mathbf{U}^F \otimes_{A \mathbf{L}^F} M ,\\
    { }^* R_{\mathbf{L} \subset \mathbf{P}}^{\mathbf{G}}: A \mathbf{G}^F\smod & \longrightarrow A \mathbf{L}^F\smod \\
    N & \longmapsto \operatorname{Hom}_{A \mathbf{G}^F}\left(A \mathbf{G}^F / \mathbf{U}^F, N\right).
  \end{aligned}
  $$
We call an $\Fr$-stable Levi subgroup \emph{$\GG$-split} if it is a Levi subgroup of an
$\Fr$-stable parabolic subgroup of $\GG$. We summarize the important properties of Harish-Chandra induction and
restriction in the following proposition whose proofs are in \cite[Section 5]{DigneMichel2020}:
\begin{prop}\label{prop:HCprop}
  \begin{enumerate}[(i)]
  \item  $R_{\mathbf{L} \subset \mathbf{P}}^{\mathbf{G}}=\operatorname{Ind}_{A
      \mathbf{P}^F}^{A \mathbf{G}^F} \circ \operatorname{Inf}_{A
      \mathbf{L}^F}^{A \mathbf{P}^F}$ (where the inflation
    $\operatorname{Inf}_{A \mathbf{L}^F}^{A \mathbf{P}^F}$ is the trivial
    natural lifting through the quotient $\mathbf{P}^F \rightarrow \mathbf{L}^F$ ).
    \item ${ }^* R_{\mathbf{L} \subset \mathbf{P}}^{\mathbf{G}}$ is $\operatorname{Res}_{A \mathbf{P}^F}^{A \mathbf{G}^F}$ followed with the taking of fixed points under $\mathbf{U}^F$.
   \item $R_{\mathbf{L} \subset \mathbf{P}}^{\mathbf{G}}$ and ${ }^*
     R_{\mathbf{L} \subset \mathbf{P}}^{\mathbf{G}}$ are exact and biadjoint,
     \emph{i.e.}
     \begin{equation}
       \begin{aligned}
         & \operatorname{Hom}_{A \mathbf{G}^F}\left(R_{\mathbf{L} \subset \mathbf{P}}^{\mathbf{G}}(M), N\right) \cong \operatorname{Hom}_{A \mathbf{L}^F}\left(M,{ }^* R_{\mathbf{L} \subset \mathbf{P}}^{\mathbf{G}}(N)\right) \\
         &   \operatorname{Hom}_{A \mathbf{G}^F}\left(N, R_{\mathbf{L} \subset \mathbf{P}}^{\mathbf{G}}(M)\right) \cong \operatorname{Hom}_{A \mathbf{L}^F}\left({ }^* R_{\mathbf{L} \subset \mathbf{P}}^{\mathbf{G}}(N), M\right)
       \end{aligned}
     \end{equation}
     for $M \in A\mathbf{L}^{\Fr}\smod$ and $N \in A\mathbf{L}^{\Fr}\smod$.
     \item (Transitivity) Let $\mathbf{Q}$ be an $\Fr$-stable parabolic subgroup contained in
       $\PP=\mathbf{L} \UU$, and $\MM$ an $\Fr$-stable Levi subgroup of $\mathbf{Q}$ contained
       in $\mathbf{L}$, then
       \[
 R^{\GG}_{\mathbf{L} \subset \mathbf{P}} \circ R^{\mathbf{L}}_{\MM \subset \mathbf{L} \cap
   \mathbf{Q}} = R^{\GG}_{\MM \subset \mathbf{Q}}, \quad { }^* R_{\mathbf{M}
   \subset \mathbf{Q} \cap \mathbf{L} }^{\mathbf{L}} \circ { }^* R_{\mathbf{L} \subset
 \mathbf{P}}^{\mathbf{G}} =  { }^* R_{\MM \subset \mathbf{Q}}^{\GG} .
\]
\item (Independence of the parabolic groups) Let $\bLL$ be a common $\Fr$-stable Levi
  component of the $\Fr$-stable parabolic
  subgroup $\PP$ and $\PP'$, then
  \[
\xHC{\GG}{\bLL}{\PP} \cong \xHC{\GG}{\bLL}{\PP'}, \quad \HC{\GG}{\bLL}{\PP}
\cong \HC{\GG}{\bLL}{\PP'}.
  \]
\end{enumerate}
\end{prop}

We remark that the assumption $p$ is invertible in $A$ is essential for (v) of
Proposition \ref{prop:HCprop}. We may omit the parabolic subgroups in the notation for Harish-Chandra induction
and restriction in $\GG$-split cases.

\begin{thm}[Mackey formula]\label{thm:Mackey}
 Let $\PP $ and $\mathbf{Q}$ be two $\Fr$-stable parabolic subgroups
  of $\GG$, and $\mathbf{L} $ (\emph{resp.} $\MM$) be an $\Fr$-stable Levi subgroup of
  $\PP$ (\emph{resp.} $\mathbf{Q}$). Then
  \begin{equation}
    \label{eq:MackeyHC}
    {}^{*}R_{\mathbf{L}}^{\GG} \circ R^{\GG}_{\MM } = \bigoplus_{x \in \mathbf{L}^{\Fr} \backslash \mathcal{S}(\mathbf{L}, \MM)^{\Fr} /\MM^{\Fr}} R_{\mathbf{L} \cap{ }^x \mathbf{M}}^{\mathbf{L}} \circ{ }^* R_{\mathbf{L} \cap{ }^x \mathbf{M} }^{{}^{x} \mathbf{M}} \circ \ad x
  \end{equation}
  with $\mathcal{S}(\mathbf{L}, \MM) = \{ x \in \GG \ | \ \mathbf{L} \cap {}^{x}
  \MM \textrm{ contains a maximal torus of } \GG\}$ and $\ad x: A \MM \smod
  \rightarrow A {}^{x}\MM \smod $ denotes the action of $x$ by conjugation on the representations.
\end{thm}
Using the same notations as the above theorem, the fact that the inclusion
$\mathcal{S} (\bLL, \MM) \hookrightarrow \GG$ induces a bijection $\bLL^{\Fr} \backslash
\mathcal{S} (\bLL, \MM)^{\Fr} / \MM^{\Fr} \cong \PP^{\Fr} \backslash \GG^{\Fr} /
\bQQ^{\Fr}$ leads to the following Corollary:
\begin{cor}\label{cor:Mackey}
  Let $I$ and $J$ be two $\Fr$-stable subsets of $S$. Then
  \begin{equation}
  {}^{*}R^{\GG}_{\bLL_I} \circ R^{\GG}_{\bLL_{J}} \cong \sum_{w \in W^{\Fr}_{I} \backslash W^{\Fr} / W_J^{\Fr}} R^{\bLL_I}_{\bLL_{I} \cap {}^w \bLL_{J}} \circ {}^{*}R^{{}^w\bLL_{J}}_{\bLL_I \cap {}^w \bLL_J} \circ \ad w .
  \end{equation}
\end{cor}

 \begin{defn}\label{defn:FinCuspidal}
  An $A \mathbf{G}^F$-module $M$ is said to be \emph{cuspidal} if ${ }^* R_{\mathbf{L}}^{\mathbf{G}}(M)=0$ for all proper $\mathbf{G}$-split Levi subgroups $\mathbf{L}$.
  \end{defn}

We denote the set of cuspidal $\GG^{\Fr}$-modules by $\Cusp_A( \GG^{\Fr})$ or
just $\Cusp (\GG^{\Fr})$ if the base ring is clear.

  \begin{defn}
A \emph{cuspidal pair} is a pair $(\mathbf{L}, \Lambda)$ where $\mathbf{L}$ is an
$\GG$-split Levi subgroup and $\Lambda$ is a cuspidal simple $A \mathbf{L}^F$-module. The
\emph{Harish-Chandra series} corresponding to such a pair defined by
$$
\operatorname{Irr}(\mathbf{G}^F |(\mathbf{L}, \Lambda))=\left\{M \in
  \operatorname{Irr} \mathbf{G}^F \mid R_{\mathbf{L}}^{\mathbf{G}}(\Lambda)
  \twoheadrightarrow M\right\} / \sim .
$$
  \end{defn}

  \begin{thm}\label{thm:HarishChandraDecomposition}
    \begin{enumerate}[(i)]
    \item For any simple $A \GG^{\Fr} \smod$ $M$ there exist a
      $\mathbf{G}$-split Levi subgroup $\mathbf{L}$ and a (simple) cuspidal
      $\mathbf{L}^F$-module $\Lambda$ such that $M$ is in the \emph{head} of
      $R_{\mathbf{L}}^{\mathbf{G}}(\Lambda)$, \emph{i.e.} such that there exists a
      $\mathbf{G}^F$-equivariant surjective map $R_{\mathbf{L}}^{\mathbf{G}}(\Lambda)
      \twoheadrightarrow  M $.
      \item If $A=K$ is a field. Let $(\mathbf{L}, \Lambda)$ and
        $\left(\mathbf{L}^{\prime}, \Lambda^{\prime}\right)$ be two cuspidal pairs.
        Then
        \[
\operatorname{Irr}_{K}(\mathbf{G}^{\Fr}|(\mathbf{L}, \Lambda)) \cap
\operatorname{Irr}_{K}(\mathbf{G}^F | \left(\mathbf{L}^{\prime},
    \Lambda^{\prime})\right) \neq \emptyset \ \Leftrightarrow \ \exists g \in
\mathbf{G}^F \textrm{ such that }\left(\mathbf{L}^{\prime}, \Lambda^{\prime}\right)={ }^g(\mathbf{L}, \Lambda).
\]
\end{enumerate}
\end{thm}
\begin{proof}
  \begin{enumerate}[(i)]
  \item  Let $\bLL$ be a minimal Levi subgroup such that ${}^{*}R^{\GG}_{\bLL} (M) \neq
  0$, by transitivity of restriction, ${}^{*}R^{\GG}_{\bLL} (M) $ is a cuspidal
  $A \bLL^{\Fr}$-module. Let $\Lambda$ be any simple submodule of
  ${}^{*}R^{\GG}_{\bLL} (M) $, we conclude from $\Hom_{A \GG^{\Fr}}
  (R^{\GG}_{\bLL}\Lambda, M) \cong \Hom_{A \bLL^{\Fr}} (\Lambda, {}^{*}R^{\GG}_{\bLL} (M)
  )$.
  \item See \cite[Theorem 5.3.7]{DigneMichel2020}.
  \end{enumerate}
\end{proof}

We conclude from the above theorem that when $A=K$ is a field, we have
\begin{thm}
  The set of irreducible representations of $\GG^{\Fr}$ is partitioned into
  Harish-Chandra series:
  \[
\operatorname{Irr}_K \mathbf{G}^F=\bigsqcup_{(\mathbf{L}, \Lambda) \in \mathcal{C}(\GG^{\Fr}) / \mathbf{G}^F} \operatorname{Irr}_K\left(\mathbf{G}^F \mid(\mathbf{L}, \Lambda)\right)
\]
where $\mathcal{C}(\GG^{\Fr})$ denote the set of cuspidal pairs of $\GG^{\Fr}$.
\end{thm}
We define the map $ \pr_{\GG^{\Fr} \mid (\mathbf{L}, \Lambda)}$ as the
projection map
\begin{equation}
  \label{eq:prfinite}
   \operatorname{Irr}_K \mathbf{G}^F \rightarrow
  \operatorname{Irr}_K\left(\mathbf{G}^F \mid ( \mathbf{L}, \Lambda ) \right).
\end{equation}

\subsubsection{Endomorphism algebras as finite Hecke algebras}\label{subsec:EndasHecke}
\noindent  From now on we work under the following assumption for the rest of this
Section:
\begin{assmp}\label{assmp:DMChp6}
  $\GG$ is a connected reductive group defined over $\Fq$, $A=K$ is a field of
  characteristic different from $p$ and $\End_{K \bLL^{\Fr}} (\Lambda) = K$.
\end{assmp}
Now we define $N_{\GG^{\Fr}} (\bLL, \Lambda)= \{  n \in N_{\GG^{\Fr}} (\bLL) \ | \
{}^{n} \Lambda \cong \Lambda \}$ and $W_{\GG^{\Fr}} (\bLL, \Lambda) =
N_{\GG^{\Fr}} (\bLL, \Lambda) / \bLL^{\Fr}$. The endomorphism algebra associated with a
cuspidal pair $(\bLL, \Lambda)$ is $\H (\bLL, \Lambda) = \End_{\GG^{\Fr}}
(R^{\GG}_{\bLL} (\Lambda))$.
From Mackey formula, we have
\[
\H (\bLL, \Lambda) = \End_{\GG^{\Fr}}(R^{\GG}_{\bLL}N) \cong \bigoplus_{w \in
  W(\bLL, \Lambda)^{\Fr}} \Hom (\Lambda, {}^{n_w} \Lambda)
\]
where $n_w$ is any representative of $w \in W(\bLL, \Lambda)^{\Fr}$ in $N(\bLL,
\Lambda)^{\Fr}$. Each $ \Hom (\Lambda, {}^{n_w} \Lambda)$ is one dimensional,
denote the generator by $\gamma_n$. Define $e_{\UU^{\Fr}}$ as
$|\UU^{\Fr}|^{-1}\sum_{u \in \UU^{\Fr}}u$, and the linear maps
\begin{equation}
  \begin{aligned}
B_w : R^{\GG}_{\bLL} & \rightarrow R^{\GG}_{\bLL} \Lambda  \\
  ge_{\UU^{\Fr}} \otimes_{K \bLL^{\Fr}} x    &   \mapsto ge_{\UU^{\Fr}}n_w^{-1}e_{\UU^{\Fr}} \otimes_{K \bLL^{\Fr}} \gamma_{n_w} (x).
  \end{aligned}
\end{equation}
We know for $w \in W_{\GG^{\Fr}} (\bLL, \Lambda)$, $B_w$ form a basis of $\H
(\bLL, \Lambda)$. From \cite[Lemma 6.18]{DigneMichel2020}, we have the following proposition
\begin{prop}
  If $\ell (w) + \ell (w') = \ell (w w')$, then $B_w B_{w'} = \lambda (w,
 w')B_{ww'}$, where $\lambda$ is a $2$-cocycle given by $\gamma_{n_w} \circ
 \gamma_{n_w'}=\lambda (w,w') \gamma_{n_w n_{w'}}$.
\end{prop}
Define the $\Hom$-functor as follows:
\begin{equation}
  \begin{aligned}
\mathcal{E}_G: K \GG^{\Fr} \smod &  \rightarrow  \H(\bLL,\Lambda)^{\operatorname{op}} \smod \\
  M & \mapsto \Hom_{\GG^{\Fr}} (R^{\GG}_{\bLL}(\Lambda), M).
  \end{aligned}
\end{equation}
We will omit the sub-index of $\mathcal{E}_G$ if no confusion is caused.
Following \cite[Proposition 10.8]{DudasMichel2015}, we have
\begin{prop}\label{prop:IrrGIsoIrrH}
  The functor $\mathcal{E}$ induces a bijection
  \[
\Irr_{K}(\GG^{\Fr}|(\bLL, \Lambda)) \leftrightarrow \Irr_{K} (\H (\bLL, \Lambda)).
  \]
\end{prop}

\begin{thm}
  Let  $\bLL \subset \MM \subset \GG$ be two split $\Fr$-stable Levi components,
  $(\bLL, \Lambda)$ be a cuspidal pair in $\GG$. Then the following two diagrams commute
   \begin{center}
    \begin{equation}
  \label{eq:LieTypeRes}
  \xymatrix@C=6.5em@R=3.5em{
    \Z \Irr_{K}(\GG^{\Fr}|(\bLL, \Lambda)) \ar[d]_{ \pr_{\MM^{\Fr}|({}^{w}  \bLL, {}^{w} \Lambda)} \circ {}^{*}R^{\GG}_{\MM}} \ar[r]^{\mathcal{E}_G} & \Z \Irr_{K}( \End_{\GG^{\Fr}}(R^{\GG}_{\bLL}\Lambda)) \ar[d]^{\Res^{\small \End_{\GG^{\Fr}} (R_{\bLL}^{\GG}\Lambda)}_{ \small \End_{\MM^{\Fr}} (R^{\MM}_{ {}^{w}  \bLL} {}^{w} \Lambda)}}\\
   \Z \Irr_{K}(\MM^{\Fr}|( {}^{w} \bLL, {}^{w} \Lambda))   \ar[r]_{\mathcal{E}_M} & \Z \Irr_K ( \End_{\MM^{\Fr}}(R^{\MM}_{{}^{w}  \bLL}{}^{w} \Lambda))
  }
\end{equation}

\begin{equation}
  \label{eq:LieTypeInd}
  \xymatrix@C=6.5em@R=3.5em{
   \Z \Irr_{K}(\GG^{\Fr}|(\bLL, \Lambda)) \ar[r]^{\mathcal{E}_G} & \Z \Irr_{K}(\End_{\GG^{\Fr}}(R^{\GG}_{\bLL}N))\\
  \Z  \Irr_{K}(\MM^{\Fr}|(\bLL, \Lambda)) \ar[u]^{R^{\GG}_{\MM}} \ar[r]_{\mathcal{E}_{M}} & \Z  \Irr_K ( \End_{\MM^{\Fr}}(R^{\MM}_{\bLL}\Lambda)) \ar[u]_{\Ind^{\small \End_{\GG^{\Fr}} (R_{\bLL}^{\GG}\Lambda)}_{ \small \End_{\MM^{\Fr}} (R^{\MM}_{\bLL} \Lambda)}}
  }
\end{equation}
\end{center}
where $w$ belongs to $ \MM^{\Fr} \backslash \{ w \in \mathcal{S} (\MM, \bLL)^{\Fr} \
| \  w (\bLL^{\Fr}) \subset \MM^{\Fr} \}/\bLL^{\Fr}$.
\end{thm}

\subsubsection{The Alvis-Curtis-Kawanaka duality}
From now on we make the following assumption
\begin{assmp}
We require  $A=K= \C$.
\end{assmp}
Let $K_0 (\C \GG^{\Fr}\smod)$ denote the Grothendieck group of the category of $\C
\GG^{\Fr}$-modules. We know this category is semisimple, a $\Z$-basis of this
group is given by complex irreducible characters. Any element of $K_0 (\C
\GG^{\Fr}\smod)$ can be described by a virtual character. In this section we will
treat the characters and the $K_0 (\C \GG^{\Fr}\smod)$ in equal footing, the
linear map induced by Harish-Chandra induction and restriction on them are
denoted by the same notations $R^{\GG}_{\bLL}$ and ${}^{*}R^{\GG}_{\bLL}$.

Let $\BB$ be a fixed $\Fr$-stable Borel, and $\mathbf{T}$ is a fixed
$\Fr$-stable maximal torus of $\BB$. Set $W = N_{\GG}(\bTT)^{\Fr} / \bTT^{\Fr}$,
then $(\BB^{\Fr}, N_{\GG}(\bTT)^{\Fr})$ is a $BN$-pair with Weyl group $W$. In the
Coxeter system $(W,S)$, denote by $W_I$ the subgroup of $W$ generated by $I$.
Recall that the standard parabolic subgroup containing $\BB^{\Fr}$ are
$\mathbf{P}_I = \BB^{\Fr}W_I \BB^{\Fr}$.
\begin{defn}\label{defn:ACKduality}
  The \emph{Alvis-Curtis-Kawanaka duality} is the linear map on
  $K_0\left(\mathbb{C} \mathbf{
    G}^F \smod\right)$ defined by
$$
D_{\mathbf{G}}=\sum_{I \subset S, \ \Fr(I)=I}(-1)^{|I/\Fr|} R_{\mathbf{L}_I}^{\mathbf{G}} \circ{ }^* R_{\mathbf{L}_{I}}^{\mathbf{G}}
$$
where  $\mathbf{L}_I$ is the standard $F$-stable Levi complement of $\Fr$-stable
standard parabolic subgroup $\mathbf{P}_I$.
\end{defn}

 The Borel subgroups, parabolic subgroups and Levi subgroups appearing in the
 following proposition are all assumed to be $\Fr$-stable. Notations will have
 their usual meaning. We list some
 important properties of $D_{\GG}$ following \cite[Section 12]{DigneMichel2020}:
\begin{prop}\label{prop:ACKduality}
  \begin{enumerate}[(i)]
  \item (C. Curtis) $D_{\mathbf{G}} \circ
    R_{\mathbf{L}}^{\mathbf{G}}=R_{\mathbf{L}}^{\mathbf{G}} \circ
    D_{\mathbf{L}}$ and ${ }^* R_{\mathbf{L}}^{\mathbf{G}} \circ D_{\mathbf{G}}=D_{\mathbf{L}} \circ{ }^* R_{\mathbf{L}}^{\mathbf{G}}$.
  \item (C. Curtis) $D_{\GG}$ is self-adjoint (acting of characters).
    \item (C. Curtis) $ D_{\GG} \circ D_{\GG} $ is the identity map on $K_0 (\C
      \GG^{\Fr}\smod)$.
    \item Let $(\mathbf{L}, \Lambda)$ be a cuspidal pair and let $\gamma $ be
      the character of some irreducible representation such that $\gamma \in
      \Irr (\GG^{\Fr} | (\bLL, \Lambda))$. Then $(-1)^{r(\mathbf{L})}
      D_{\mathbf{G}}(\gamma) \in \operatorname{Irr}\left(\mathbf{G}^F
        \mid(\mathbf{L}, \Lambda )\right)$, where $r(\mathbf{L})$ is the
      semisimple rank of $\bLL$.
  \end{enumerate}
\end{prop}
We see from $(iv)$ of the above proposition  that $D_{\GG}$ preserves the
$\Irr(\GG^{\Fr}|(\bLL, \Lambda))$ up to $\pm 1$.
\subsection{A review of Howlett-Lehrer's work}
In this section we make the following assumption
\begin{assmp}\label{assmp:split}
  We require $A=K= \C$, $(\GG, \Fr)$ is \emph{split}, \emph{i.e.} $\Fr$ acts trivially on $W$.
\end{assmp}

In the case of $\C$-modules are considered, the quadratic relations for $B_w$
are explicit and the $2$-cocycle $\lambda$ is trivial. There exists a set of
involutions $S(\bLL, \Lambda) \subset W(\bLL, \Lambda)^{\Fr}$ (defined in
section \ref{subsec:EndasHecke}) such that
$(S(\bLL, \Lambda), W(\bLL, \Lambda)^{\Fr})$ is a Coxeter system.
Combining Lemma 4.2 and Theorem 5.9 of \emph{loc.cit.}, we have the following theorem:
\begin{thm}\label{thm:HeckeIsoEnd}
  The endomorphism algebra $\End_{\GG^{\Fr}}(R^{\GG}_{\bLL}N)$ is a Hecke
  algebra, we have
  \begin{equation}
    \label{eq:HLHecke}
    \H_{\mathbf{q}} (S(\bLL, \Lambda), W(\bLL, \Lambda)^{\Fr}) \otimes_{f} \C \cong \End_{\GG^{\Fr}}(R^{\GG}_{\bLL}N)
  \end{equation}
  where $f$ denotes a specialization defined in \cite[Lemma 4.2]{HowlettLehrer1983}.
\end{thm}

Furthermore, in characteristic $0$ case, by using Tits' deformation  theorem we
can sharpen Proposition \ref{prop:IrrGIsoIrrH} and Theorem
\ref{thm:HeckeIsoEnd}:
\begin{thm}[Howlett-Lehrer]\label{thm:HLComparison}
  Let $\GG^{\Fr}$ be a finite group of Lie type and $(\bLL, \Lambda)$ be a
  cuspidal pair, then
  \begin{equation}
    \label{eq:HLComparison}
   \C  W(\bLL, \Lambda)^{\Fr} \cong   \H_{\mathbf{q}} (S(\bLL, \Lambda), W(\bLL, \Lambda)^{\Fr}) \otimes_{\substack{q_{s} \mapsto 1}} \C \cong \H_{\mathbf{q}} (S(\bLL, \Lambda), W(\bLL, \Lambda)^{\Fr}) \otimes_{f} \C \cong \End_{\GG^{\Fr}}(R^{\GG}_{\bLL}N) .
 \end{equation}
 In particular, we have
 \[
\Irr_{\C} ( W(\bLL, \Lambda)^{\Fr}) \leftrightarrow \Irr_{\C} (\GG^{\Fr}|(\bLL, \Lambda)).
 \]
\end{thm}
Moreover, in this setting the Alvis-Curtis-Kawanaka duality (Definition \ref{defn:ACKduality}) becomes
\[
D_{\mathbf{G}}=\sum_{I \subset S}(-1)^{|I|} R_{\mathbf{L}_I}^{\mathbf{G}} \circ{ }^* R_{\mathbf{L}_{I}}^{\mathbf{G}}.
\]

We now begin to introduce the setup in Howlett-Lehrer's series of works (\cite{Howlett1980},
\cite{HowlettLehrer1982}, \cite{HowlettLehrer1983}), where mainly complex
representations are considered. The finite group $G=\GG^{\Fr}$ ($\GG(\Fq)$ in
Howlett-Lehrer's work) has a $(B,N)$-pair with $B=P_0 = \PP_0^{\Fr}$, $N=
N_{\GG}(\bTT)^{\Fr}$ where  $\PP_0$ is a fixed Borel $\Fq$-subgroup and $\bTT$
is a fixed maximal $\Fq$-split torus of $\PP_0$. The Borel subgroup determines a
set of simple roots, a
positive system $\Delta \subset \Sigma^{+} \subset \Sigma= \Phi (\bTT^{\Fr},\GG^{\Fr})$ of $W$.  The standard parabolic
subgroups $P_I$ ($=BW_IB$)of $G$ are in bijection with subsets $I \subset \Delta$, and can be
expressed as follows:
\[
P_I=M_I U_I \quad \textrm{ where }\left\{\begin{array}{l}
M_I=\left\langle T, U_a \mid a \in \Sigma_I\right\rangle \\
U_I=\left\langle U_a \mid a \in \Sigma^{+} \backslash \Sigma_I\right\rangle
\end{array}\right.
\]
where $\Sigma_I$ is the sub-root system of $\Sigma$ spanned by $I$.

From now on until the end of this Section, we fix a
subset $I_0 \subset \Delta$, and set $W_{I_0}$, $L_{I_0}$, $U_{I_0}$ and
$P_{I_0}$ as above. Let $D$ be a representation of any Levi $M_I$, we will use the same
notation $D$ to denote the lift from $M_I$ to $P_I$, for $w \in W$, ${}^{w}D
(-)$ and ${}^{w}P_{I}$ are always understood as ${}^{n_{w}}D
(-)$ and ${}^{n_w}P_{I}$ where $n_w$ is any representative of $w$ in $N$. Let
$T_0=T_{I_0}$ denote the maximal split torus contained in $Z(L_{I_0})$. We use
$\Phi_{I_0}$ to denote the sub-root system spanned by $I_0$. The pair $(B \cap
L_{I_0}, N \cap L_{I_0})$ provides $BN$-pair for $L_{I_0}$. We use $\Lambda$ to denote an irreducible cuspidal representation
of $L_{I_0}$, whose character is $\chi_{\Lambda}$.

\begin{prop}
  Denote by $\Aut (T_0)$ automorphism group of $T_0$ induced by conjugation of
  $G$, we have $\Aut(T_0) \cong N_{W}(W_{I_0})/W_{I_0} \cong S_{I_0} : =\{ w \in
  W \ | \ w I_0 =I_0\}$.
\end{prop}
\begin{defn}\label{defn:RamificationGroup}
  The \emph{ramification group} associated to
  cuspidal pair $(L_{I_0}, \Lambda)$ is defined as follows
  \[
 W(I_0, \Lambda) = \{ w \in S_{I_{0}} \ | \ {}^{w} \Lambda =  \Lambda \}.
\]
When the set $I_0$ is clear, we may denote $W(I_0, \Lambda)$ shortly as $W(\Lambda)$.
\end{defn}
\begin{rem}
 $W(\Lambda)$ agrees with previously defined $W_{\GG^{\Fr}} (\bLL, \Lambda)$
 under the split Assumption \ref{assmp:split}.
\end{rem}
Let $\hat{\Omega}(I_0):= \{ a \in \Sigma \ | \  w (I_0 \cup \{a \} ) \subset \Delta
  \textrm{ for some }w \in W \}$. For an element $a \in \hat{\Omega}(I_0)$, we
  define $v[a, I_0]=ut$ where $u$ is the longest element of $W_{I_0 \cup \{ a \}}$,
  and $t$ is the longest element of $W_{I_0}$. We introduce the following definition
  from \cite[Definition 3.6]{HowlettLehrer1983}.
\begin{defn}[Howlett-Lehrer]
  Let $\Gamma (I_{0}, \Lambda)$ be the set of roots $a \in \Sigma \backslash
  I_0$ satisfying
  \begin{enumerate}[(i)]
  \item $w (I_0 \cup \{ a \}) \subset \Delta$ for some $w \in W$.
  \item $v [a, I_0] \in W(I_0, \Lambda)$.
    \item With $w$ as in $(i)$, $I' = w(I_0 \cup \{ a \}) $, the induced
      representation $\Ind_{P_{wI_0} \cap M_{I'}}^{M_{I'}} ({}^w \Lambda)$ has
      exactly two irreducible constituents $\rho_{1}$ and $\rho_2$ (we may
      assume $\dim (\rho_2) \geq \dim (\rho_1)$), and $p_a = \frac{\dim
        (\rho_{2})}{ \dim (\rho_1)} >1$.
    \end{enumerate}
    We denote by $\Gamma^{+} (I_{0}, \Lambda)$ the set $\Gamma (I_{0}, \Lambda)
    \cap \Sigma^{+}$, and denote by $R(I_0, \Lambda)$ the subgroup of $W
    (I_0, \Lambda)$ generated by $\{ v[a, I_0] \ | \ a \in \Gamma (I_0,
    \Lambda) \}$. For any $w \in W$, set $N(w):= \{ a \in \Sigma^{+} \ | \ wa \in \Sigma^{-} \}$. We define $\Delta (I_0, \Lambda)$ as the set
    \[
      \{ a \in
    \Gamma^{+} (I_{0}, \Lambda) \ | \ N(v[a, I_0]) \cap \Gamma (I_{0}, \Lambda)
    = \{ a \}\}.
  \]
  We set $S(I_0, \Lambda) := \{ v[a, I_0] \ | \ a \in \Delta (I_0, \Lambda) \}$.
We may abbreviate $\Gamma (I_{0}, \Lambda)$ as $\Gamma (\Lambda)$  when no
    confusion is caused; we will use similar abbreviations for $R(I_0, \Lambda)$, $S(I_0, \Lambda) $ and
     $\Delta (I_0, \Lambda)$.
  \end{defn}
  We have the following fact from \cite[Lemma 2.7]{Howlett1980} and \cite[Lemma 3.9]{HowlettLehrer1983}:
  \begin{lem}\label{lem:W=CR}
    \begin{enumerate}[(i)]
    \item  The group $R(\Lambda)$ is a normal subgroup of $W(\Lambda)$ and acts as
      reflection group on $I_0^{\bot}$.
      \item The projection of $\Gamma (\Lambda)$ to  $I_0^{\bot}$ is a root
        system for $R (\Lambda)$. The projection of $\Gamma^{+} (\Lambda)$  to
        $I_0^{\bot}$ is a positive system and the projection of $\Delta
        (\Lambda)$ is the corresponding fundamental system.
      \item We can write $W(\Lambda)$ as $C (\Lambda) \ltimes R(\Lambda)$ where
        $C(\Lambda) = \{ w \in W(\Lambda) \ | \ w \Gamma^{+}( \Lambda)  = \Gamma^{+}( \Lambda)  \}$.
    \end{enumerate}
  \end{lem}
  \begin{defn}
    \label{defn:lIbotpw}
   \begin{enumerate}[(i)]
   \item  We define the length function $\ell_{I_0^{\bot}}$ for $w \in W(\Lambda)$ as
    the length function after projection to $\langle I_{0} \rangle^{\bot}$,
    \emph{i.e.} for $w =w_1 w_{2}$ with $w_1 \in C(\Lambda)$, $w_2 \in
    R(\Lambda)$, $\ell_{I_0^{\bot}} (w) = \ell_{R(\Lambda)} (w_2)$.
    \item (\cite[Definition 4.9]{Howlett1980}) For any $w \in W(\Lambda)$, we
      define
      \[
p_{w} = \prod_{a \in N(w) \cap \Gamma (\Lambda)} p_{a}.
      \]
   \end{enumerate}
  \end{defn}

  We can now state the main theorem of \cite[Theorem 4.14]{Howlett1980}
  \begin{thm}[Howlett-Lehrer]
  The algebra $E_G(\Lambda) = \End_G (\Ind_{P_{I_0}}^{G} (\Lambda))$ has a
  $\mathbb{C}$-basis (see \cite[Section 4]{Howlett1980}) $\left\{T_w \mid w \in W(\Lambda)\right\}$ whose
  multiplication table satisfies (i) to (iv) below for all $w \in W(\Lambda), x
  \in C(\Lambda)$ and $a \in \Delta$. We set $v=v[a, I_0]$. Then we have
  \begin{enumerate}[(i)]
  \item
    \[
      T_w T_x=\mu(w, x) T_{w x},
      \]
\item \[
      T_x T_w=\mu(x, w) T_{x w},
    \]
  \item \[
      T_v T_w= \begin{cases}T_{v w} & \left(\text { if } w^{-1} a \in \Gamma^{+}(\Lambda) \right), \\
p_a T_{v w}+\left(p_a-1\right) T_w & \left(\text { if } w^{-1} a \notin
                                     \Gamma^{+}(\Lambda) \right),\end{cases}
                               \]
                             \item
                               \[
                                 T_w T_v= \begin{cases}T_{w v} & \left(\text { if } w a \in \Gamma^{+}(\Lambda) \right), \\
p_a T_{w v}+\left(p_a-1\right) T_w & \left(\text { if } w a \notin
                                     \Gamma^{+}(\Lambda) \right) .\end{cases}
                                 \]
  \end{enumerate}

The $2$-cocycle $\mu$ of $W(\Lambda)$  appearing in (i) and (ii)  satisfies the
properties that
\begin{enumerate}[(a)]
\item $\mu(x v, y w)=\mu(x, y)$ for all $x, y \in C(\Lambda), v, w \in
  R(\Lambda)$,
\item the value $\mu (x,y)$ depends only on the cohomological class of $\mu$,
  \item $\mu(x, 1)=\mu(1, x)=1$ for all $x \in C(\Lambda)$.
\end{enumerate}
  \end{thm}
We will work under the assumption that $\mu$ is trivial which
is proved by Lusztig-Geck (\cite{Lusztig1984} and \cite{Geck1993}) for finite
groups of Lie type. We have the following
\begin{cor}\label{cor:EGLbasis}
  The endomorphism algebra $E_G(\Lambda)$ has the basis consisting of
  $T_w = T_{r(w)}T_{c(w)}$ for all $w =c(w) r(w) \in W(\Lambda)$, where $c(w) $
  (\emph{resp.} $ r(w)$)
  denotes the $C(\Lambda)$ (\emph{resp.} $R(\Lambda)$) component in the
  decomposition $W(\Lambda) = C(\Lambda) \ltimes R(\Lambda)$.
\end{cor}

In the following sections, we will need study the structure of $\End_{L_{I}}
(\Ind^{L_{I}}_{{}^w L_0}{}^w \Lambda) $ for different $I \subset \Delta$. Thus
we need a description of the structure of the group $W(\Lambda ) \cap W_I^w$ for
$w \in V_{\Lambda}$ where $V_{\Lambda}:= \{ w \in W \ | \  wI_0 \subset \Delta , \ w \Gamma^{+ } (\Lambda) \subset \Sigma^{+}\}$.

We state the following lemma (\cite[Lemma 3.13]{HowlettLehrer1983}):
\begin{lem}\label{lem:structureWLcapWwI}
  Let $w \in V_{\Lambda}$, $wI_0 \subset I \subset \Delta$, we denote by $ R(\Lambda)_{w,I}$ be the parabolic subgroup of $R (\Lambda)$
generated by $S(\Lambda)_{w, I}:=\{ v[a, I_{0}] \ | \ a \in \Delta
(\Lambda)_{w,I} \}$, where $\Delta(\Lambda)_{w,I} = \{ a \in \Delta (\Lambda), wa \in \langle I \rangle \}$. Then we have the following semidirect product:
  \[
W( \Lambda) \cap W^{w}_{I} =(W^w_I \cap C (\Lambda)) \ltimes R(\Lambda)_{w, I}.
  \]
\end{lem}
For $w \in W$ such that $w I_0 \subset \Delta$, we define the intertwining operator $B_{\Lambda, w}:
\Ind^G_{P_{I_0}} (\Lambda) \rightarrow  \Ind^G_{{}^w P_{I_0}} ({}^w\Lambda)$ by
\[
(B_{\Lambda, w} f) (x) = \operatorname{Card}(U_{wI_0})^{-1} \sum_{y \in
  U_{wI_0}} f(n_w^{-1}yx).
\]
From \cite[Lemma 3.11 and Lemma 3.17]{Howlett1980}, we know $B_{\Lambda,w}$ is a
$\C G$-isomorphism. For $w \in V_{\Lambda}$, we define a map $\tau_w: E_G (\Lambda) \rightarrow E_G ({}^w \Lambda)$ by
$\tau_w (T)= B_{\Lambda, w}TB_{\Lambda,w}^{-1}$.

We can now state the following theorem from \cite[Theorem
3.16]{HowlettLehrer1983} describing an important subalgebra of $E_{G}(\Lambda)$:
\begin{thm}\label{thm:HL1983Thm3.16}
   Let $w \in V_{\Lambda}, w I_0 \subset I \subset \Delta $. The image of $
   E_{I}({}^{w} \Lambda)
   := \End_{P_J}\left(\operatorname{Ind}_{{}^{w}P_{I_0}}^{P_I}({}^{w}\Lambda)\right)$
   under the canonical injection $E_{I}({}^{w} \Lambda)  \rightarrow
   E_G({}^{w}\Lambda)$ is the $\mathbb{C}$-linear span of $\left\{T_v^{\prime}
     \mid v \in W_I \cap W({}^{w} \Lambda)\right\}$. Furthermore the $\mathbb{C}$-linear span of $\left\{T_v \mid v \in\right.$ $\left.W_I^w \cap W(\Lambda)\right\}$ is a subalgebra of $E_G(\Lambda)$ which is mapped onto the image of $E_I({}^{w}\Lambda)$ by the isomorphism $\tau_w$.
\end{thm}

The following theorem is a motivation for what follows:
\begin{thm}[L. Solomon]\label{thm:Solomonthm1966}
  Let $\chi$ be a character of
$W$. We have
  \begin{equation}
    \label{eq:Solomonthm2}
    \sum_{I \subset S} (-1)^{|I|} \Ind_{W_I}^W \Res_{W_I}^W (\chi) = \widehat{\chi}
  \end{equation}
where $\widehat{\chi}$ is defined by $\chi (w) = \det(w) \cdot \chi (w)$.
\end{thm}
Because $\det (w) =  (-1)^{\ell (w)}$ we may see the right
hand side of \refeq{eq:Solomonthm2} as a character taking value $(-1)^{\ell (w)}
\chi ((w^{-1})^{-1})$ at $w$. Written in this way, and compared with $T^{*}_{w} =
(-1)^{\ell (w)}q(w)T_{w^{-1}}^{-1}$, the   $ \ \widehat{ {}} \ $ operation can
be seen as the prototype of ${}^{*}$ for characters of $W$.

 In the article  \cite[Theorem and Corollary 1]{HowlettLehrer1982}, R.B. Howlett and G.I. Lehrer proved a stronger version
  for the normalizer $N_{W}(W_{I_0})$ of a parabolic subgroup associated with a subset
  $I_0 \subset S$ of $W$ and its
  subgroups especially the ramification group $W(\Lambda)$.

\begin{thm}[R.B. Howlett, G.I. Lehrer]
  \label{thm:HowlettLehrer1982}
  Let $\chi$ be a character of $W(\Lambda )$, we have the following equation of characters:
  \begin{equation}
    \label{eq:HowlettLehrerCor1}
    \sum_{I \subset S} (-1)^{|I|} \sum_{w \in W_I  \backslash C_{I_0}(I) /W(\Lambda) } \Ind^{W(\Lambda)}_{W(\Lambda) \cap  W_I^w} (\Res^{W(\Lambda)}_{W(\Lambda) \cap  W_I^{w}} (\chi))  = \widehat{\chi}:= (-1)^{|I_0|} (-1)^{\ell_{I_0^{\perp}}(-)} \chi
  \end{equation}
where $I_0^{\perp}$ is its orthogonal
  complement of $I_0$. For any $I \subset S$, $C_{I_0}(I) $ is defined as the set $  \{ w
  \in W \ | \ w I_0 \subset  \langle  I \rangle  \}$.
\end{thm}
We include a detailed proof of this Theorem in Section
\ref{subsec:ProofHLAppendix} in the Appendix \ref{app:A}.
\begin{rem}
  We use the same notation as \cite{HowlettLehrer1983}, which is different from
  \cite{HowlettLehrer1982} up to an inverse.
\end{rem}

If we take $I_0= \emptyset$ and replace $W(\Lambda)$ by $W$, we see that
\refeq{eq:Solomonthm2} is a special case of \refeq{eq:HowlettLehrerCor1}.

\section{Restriction of the ACK duality to a Harish-Chandra series}\label{sec:ResDualityFinite}
We work under the same assumption as previous section. Let us fix a cuspidal
pair $(L_{I_0}, \Lambda)$ with $L_{I_0}$ a fixed standard Levi subgroup, and
choose $\pi \in \Irr_{\C}(G^{\Fr}|(L_{I_0}, \Lambda))$. We abbreviate $L_{I_0}$
by $L_0$. We study the restriction of Alvis-Curtis-Kawanaka (ACK) duality $\ND_{G}=\sum_{I \subset S}(-1)^{|I|}
R_{L_I}^{G} \circ { }^*
R_{L_{I}}^{G}$ to $\Irr_{\C}(G^{\Fr}|(L_{0}, \Lambda))$ by let it act on $\pi$. We know $\pi$ is a simple
quotient of $R^{G}_{L_0}(\Lambda)$: $R^{G}_{L_0} (\Lambda)
\twoheadrightarrow \pi$. Now applying first ${}^{*}R^{G}_{L_I}$, by Corollary
\ref{cor:Mackey} we have $ { }^*
R_{L_{I}}^{G} \circ R^{G}_{L_{0}} (\Lambda) $
equals a sum of the form $R^{L_I}_{L_{I} \cap {}^{w}L_0} \circ
{}^{*}R_{L_{I} \cap {}^{w}L_0}^{{}^{w}L_{0}} \circ \ad (w) (\Lambda) $ for $w \in
W_I \backslash W /W_{I_0}$. Then ${}^{*}R^{G}_{L_I} (\pi)$ is quotient
of $   \sum_{w \in  \mathfrak{S}_I }R^{L_I}_{ {}^{w}L_0}  ({}^{w}\Lambda)$, where $\mathfrak{S}_I  := \{ w \in W_I \backslash W /W_{I_0} \
| \  w (I_0) \subset I \}$. We conclude that the images of $\pi$ under
${}^{*}R^{G}_{L_I} $ are contained in the subsets $ \Irr_{\C}(L_I|( {}^{w} L_0, {}^{w} \Lambda))$ for $w \in  \mathfrak{S}_I $.
We now adapt the diagrams
\refeq{eq:LieTypeRes} and \refeq{eq:LieTypeInd} accordingly for each $w$ in this
set:
 \begin{center}
    \begin{equation}
  \label{eq:LieTypeResACK}
  \xymatrix@C=6.5em@R=3.5em{
    \Z \Irr_{\C}(G|(L_0, \Lambda)) \ar[d]_{ \pr_{L_I|( {}^{w} L_0, {}^{w} \Lambda)} \circ {}^{*}R^{G}_{L_I}} \ar[r]^{\mathcal{E}_G} & \Z \Irr_{\C}( \End_{G}(R^{G}_{L_0}\Lambda)) \ar[d]^{\Res^{\small \End_{G} (R_{L_0}^{G}\Lambda)}_{ \small \End_{L_I} (R^{L_I}_{ {}^{w} L_0} {}^{w} \Lambda)}}\\
   \Z \Irr_{\C}(L_I|( {}^{w} L_0, {}^{w} \Lambda))   \ar[r]_{\mathcal{E}_M} & \Z \Irr_{\C} ( \End_{L_I}(R^{L_I}_{{}^{w} L_0}{}^{w} \Lambda))
  }
\end{equation}

\begin{equation}
  \label{eq:LieTypeIndACK}
  \xymatrix@C=6.5em@R=3.5em{
   \Z \Irr_{\C}(G|(L_0, \Lambda)) \ar[r]^{\mathcal{E}_G} & \Z \Irr_{\C}(\End_{G}(R^{G}_{L_0} \Lambda))\\
  \Z  \Irr_{\C}(L_I|({}^{w} L_0, {}^{w}\Lambda)) \ar[u]^{R^{G}_{L_I}} \ar[r]_{\mathcal{E}_{M}} & \Z  \Irr_{\C} ( \End_{L_I}(R^{L_I}_{{}^{w} L_0} {}^{w} \Lambda)) \ar[u]_{\Ind^{\small \End_{G} (R_{L_0}^{G}\Lambda)}_{ \small \End_{L_I} (R^{L_I}_{{}^{w} L_0} {}^{w} \Lambda)}}
  }
\end{equation}
\end{center}
where $\pr_{G|(  L,  N)}$ denote the projection functor from $\Irr_{\C} (G)$ to
  the Harish-Chandra series $\Irr_{\C} (G | (  L,  N))$, the explicit formula is
  given in \cite{Aubert1992} and \cite{Aubert1993}. We use Theorem \ref{thm:HarishChandraDecomposition} (ii) to identify $
\Irr_{\C}(G|({}^{w} L_0, {}^{w}\Lambda)) $ with $\Irr_{\C}(G|(L_0, \Lambda))$.
We have shown in the above discussion that the map ${}^{*}R_{L_I}^G$ restricts to $\Irr_{\C}(G|(L_0, \Lambda)) $ can be written as
sum
\[
  \sum_{w \in \mathfrak{S}_I} \pr_{L_I|( {}^{w} L_0, {}^{w} \Lambda )} \circ
  {}^{*} R^{G}_{L_I}.
\]
The duality functor has the form
\begin{equation}
  \label{eq:DGHC}
   \ND_G =\sum_{I \subset S}
(-1)^{|I|} (\sum_{w \in \mathfrak{S}_I} R^{G}_{L_{I}} \circ \pr_{L_I|( {}^{w} L_0, {}^{w} \Lambda)} \circ {}^{*}R^{G}_{L_I}).
\end{equation}
Comparing with the formula in Theorem \ref{thm:HowlettLehrer1982} and using the
bijection
\[
  \Irr_{\C} ( W( \Lambda)) \leftrightarrow \Irr_{\C} (G|(L_0,
  \Lambda)),
\]
we see that \refeq{eq:DGHC} is the counterpart of left hand side of
\refeq{eq:HowlettLehrerCor1} in Theorem \ref{thm:HowlettLehrer1982} on the
finite group side. By the two commutative diagrams \refeq{eq:LieTypeRes} and \refeq{eq:LieTypeInd}, we find on the Hecke algebra side, the involution is
\begin{equation}
  \label{eq:DHHC}
\ND_{\H} =\sum_{I \subset S} (-1)^{|I|} (\sum_{w \in \mathfrak{S}_I}
\Ind^{\small \End_{G} (R_{L_0}^{G}\Lambda)}_{ \small \End_{L_I} (R^{L_I}_{{}^{w}L_0} {}^{w} \Lambda )}
\circ \Res^{\small \End_{G} (R_{L_0}^{G}\Lambda)}_{      \small \End_{L_I} (R^{L_I}_{ {}^{w} L_0} {}^{w} \Lambda)} ).
\end{equation}
We will study the explicit expression for it in the next section.

\subsection{The involution for the endomorphism algebras}\label{subsec:HLanalogue}

The following is an analogue of Theorem \ref{thm:HowlettLehrer1982} for endomorphism
algebras of induced modules whose proof is inspired by Theorem \ref{thm:Kato.main}
and uses properties from articles \cite{Howlett1980}, \cite{HowlettLehrer1982}
and \cite{HowlettLehrer1983} by R.B. Howlett, G.I. Lehrer. Recall from
\ref{lem:W=CR} that the ramification group $W(\Lambda)$ has a decomposition
\[
  W(\Lambda) = C (\Lambda) \ltimes R(\Lambda),
\]
where
\[
  C(\Lambda) = \{ w \in W(\Lambda) \ | \ w \Gamma^{+}( \Lambda)  = \Gamma^{+}(
  \Lambda)  \}.
  \]
\begin{assmp}\label{assmp:CLambdaTrivial}
  The group $C(\Lambda)$ is trivial.
\end{assmp}
This assumption is satisfied, for example, in \cite[Example 4.15]{Howlett1980}. We have our second
main Theorem as follows:
\begin{thm}\label{thm:HLanalogue}
 We work under the Assumption \ref{assmp:CLambdaTrivial}. Let $M^{*}$ denote the module $M$ endowed with the twisted action of $E_G(\Lambda)$
 defined by
 \[
   T_w^{*}=(-1)^{|I_0|+\ell_{I_0^{\perp}}(w)} p_w T_{w^{-1}}^{-1},
 \]
 where the $p_w$ is defined in Definition \ref{defn:lIbotpw}. Then we have the following equality in the Grothendieck group of $E_G(\Lambda)$-modules:
  \begin{equation}
    \label{eq:thm:HLanalogue}
    \sum_{I \subset S} (-1)^{|I|} \sum_{w \in W_I \backslash C_{I_0}(I) /W(\Lambda)} [\Ind_{E'_I}^{E_G(\Lambda)}[\Res_{E'_I}^{E_G(\Lambda)}(M)]] = [M^{*}],
  \end{equation}
  where  $E_G(\Lambda)$ is defined as $\End_G(\Ind_{P_{I_0}}^G (\Lambda))$ with $\Lambda$ viewed as a $P_{I_0}$-module
  through the natural lifting $P_{I_0} \rightarrow L_{I_0}$ and $E'_I$ is
  the subalgebra of $E_G(\Lambda)$ spanned by $\{ T_{v} \ | \ v \in W(\Lambda) \cap
   W_I^w \}$ (See Theorem \ref{thm:HL1983Thm3.16}).
\end{thm}

If we let $I_0 = \emptyset$, then $C_{I_0}(I) = W$. In this case, $P_{I_0} = P_{\emptyset}
= P_0$ is the minimal parabolic subgroup and $L_{I_0} = T$. We take $\Lambda$ to be
trivial character of $T$ then $W(1) = W$, and we have $E_G(1) = \End_G
(\Ind_{P_0}^G (1)) $ is just the finite Hecke algebra $H$ associated with $(W,S)$, and $E'_I$ becomes the subalgebra of $H$ spanned by $\{ T_{v} \ | \  v
\in W_I \}(=H_I)$. We see that \refeq{eq:thm:HLanalogue} degenerates to
\refeq{thm:Kato.main} for finite Hecke algebra.

Let us denote $V$ the real vector space spanned by simple roots $S$. By
Bourbaki \cite[Section IV]{BourbakiLie456En} (or see Theorem \ref{thm:SphereHomology} in
Appendix \ref{app:A}) we know that the Coxeter complex of $W$ is the
simplicial subdivision of the unit sphere $S(V)$. The simplexes therefore
correspond to $W$-translations of the fundamental chamber, \emph{i.e.} they are the collection $\{ wC_{I} \ | \ w
\in W \}$, where

\[
C_I = \{ v \in V \ | \  (v, \alpha) >0 \textrm{ for } \alpha \in S - I, \
(v, \alpha) =0 \textrm{ for } \alpha \in I \}.
\]
There are a lemma and a corollary from \cite{HowlettLehrer1982} which imply the
geometry we need to deal with. For a more detailed treatment, please see the
Appendix \ref{app:A}.

Let $C_{I_0}(I) $ be defined as the set $  \{ w   \in W \ | \ w I_0 \subset  \langle  I \rangle  \}$.
\begin{lem}[Howlett-Lehrer]
  \label{lem:HowlettLehrer1982Lem1}
  The set $R_{I_0} = \{ wC_{I} \ | \ I \subset S, w \in C_{I_0}(I) \}$ is precisely the set of
  those $W$-regions contained in $I_0^{\perp}$.
\end{lem}

\begin{cor}
  \label{cor:HowlettLehrer1982Cor1}
  The poset $\Gamma_{I_0} = \{ wW_{I} \ | \ I \subset S, w \in C_{I_0}(I) \}$ defines a
  subcomplex of the Coxeter complex $\Gamma$ of $W$, which is a simplicial
  subdivision of $S(I_0^{\perp})$.
\end{cor}

Further, we know that $C_{I_0}(I)$ can be seen as a union of $(N_W(W_{I}), N_W(W_{I_0}))$
double cosets, hence a union of $(W_I, W(\Lambda))$ double cosets. From
\cite[Corollary 5.5]{HowlettLehrer1983}, we know each $(W_I, W(\Lambda))$ double
coset contains an element $w$ with $wI_0 \subset I$ and $w \in V_{\Lambda}$. We
will assume the choice of such $w$ in the proof.

We see that the left hand side of \refeq{eq:thm:HLanalogue} is very close to (but
still different from) the
left hand side of \refeq{eq:DHHC}. Namely, $ \mathfrak{S}_I$ of
\refeq{eq:DHHC} by definition equals $ W_I \backslash C_{I_0}(I) /  W_{I_0}$;
the big endomorphism algebra $\End_{G} (R_{L_0}^{G}\Lambda) $  in
\refeq{eq:DHHC} equals $E_G(\Lambda)$ in  in \refeq{eq:DHHC}, and $\End_{L_I}
(R^{L_I}_{{}^{w}L_0} {}^{w} \Lambda)$ \refeq{eq:DHHC} differs slightly  from
$   \End_{P_J}\left(\operatorname{Ind}_{{}^{w}P_{I_0}}^{P_I}({}^{w}\Lambda)\right)$
which is isomorphic to $E'_{I}$ in \refeq{eq:thm:HLanalogue}.

We can now prove the Theorem \ref{thm:HLanalogue}.

\begin{proof}
  We shall imitate the proof by S-I. Kato to prove this theorem.  We work under the assumption that $C(\Lambda)$
is trivial. Under this assumption, we have $W(\Lambda) = R(\Lambda)$, thus the endomorphism algebra $E_G(\Lambda)$ is
isomorphic to $H (R(\Lambda), S(\Lambda))$ where  \[
      S(\Lambda)=\{ v[a, I_0]  \ | \ a \in
    \Gamma^{+} (I_{0}, \Lambda), \ N(v[a, I_0]) \cap \Gamma (I_{0}, \Lambda)
    = \{ a \}\}.
  \]
Moreover, we have  $W( \Lambda) \cap W^{w}_{I} = R(\Lambda)_{w,
  I}$ (see Lemma \ref{lem:structureWLcapWwI}) which is a standard parabolic
subgroup of $R(\Lambda)$. By definition,
  we have
\begin{equation}
  \begin{aligned}
    & \sum_{w \in W_I \backslash C_{I_0}(I) /W(\Lambda)}    \Ind_{E'_I}^{E_G(\Lambda)}(\Res_{E'_I}^{E_G(\Lambda)}(M)) \\
    & = \sum_{w \in W_I \backslash C_{I_0}(I) /W(\Lambda)}    E_G(\Lambda) \otimes_{E'_{I}} M \\
    & = \begin{cases}
      0, \quad \textrm{ (in this case } C_{I_0} (I) = \emptyset) & \textrm{ if }|I| < |I_0|,\\
    \displaystyle \sum_{w \in W_I \backslash C_{I_0}(I) /W(\Lambda)} E_G(\Lambda) \otimes_{K} M / \sum_{s \in S(\Lambda)_{w, I}} \langle \displaystyle  hT_u \otimes \pi (T_u)^{-1}m -h \otimes m  \rangle , & \textrm{ if } |I| \geq |I_0|.
    \end{cases}
  \end{aligned}
\end{equation}
To simplify notation, we will denote $L_s = hT_s \otimes \pi (T_s)^{-1}m -h
\otimes m $.

Let us define
\[
   {}_{w}\pi^{I'}_{I} :  E_G(\Lambda) \otimes_K M / \sum_{s \in S(\Lambda)_{w, I}} L_s
   \rightarrow   E_G(\Lambda) \otimes_K M / \sum_{s \in S(\Lambda)_{w, I'}} L_s
\]
as the natural projection where $I'$ is
obtained by adding one more element $wb$ from the subset $w (\Delta (\Lambda) \backslash \Delta
(\Lambda)_{w,I}) \subset \Delta$. Here ${}_{w} \pi^{I'}_I$ is an
$E_G (\Lambda)$-homomorphism. We also define

\[
  \epsilon^{I'}_I : \bigwedge^{|I|-|I_0|} (K^{I \backslash w I_0}) \rightarrow
  \bigwedge^{|I'|-|I_0|} (K^{I'\backslash w I_0})
\]
as the natural isomorphism given by $v \rightarrow  v \wedge c$, $c \in I'
\backslash I$. Here $K^X$ is the free $K$-modules for any set $X$, and put
$\bigwedge^{|I|}(K^I) = K$ for $I = \emptyset$.

We consider the following complex:

\begin{equation}
  \label{eq:Katoeq1.2again}
  0 \rightarrow C_0 \xrightarrow{d_0} C_1 \xrightarrow{d_{1}}  \ldots  \xrightarrow{d_{i-1}} C_i  \xrightarrow{d_{i}} \ldots  \xrightarrow{d_{|S|-1}} C_{|S|} \rightarrow 0
\end{equation}
where
\[
  C_i = \bigoplus_{\substack{ |I|=i,\\ w \in W_I \backslash C_{I_0}(I) /W(\Lambda)}} (E_G(\Lambda) \otimes M /  \sum_{s \in S( \Lambda)_{w,
      I} } L_s) \otimes
  \bigwedge^{i-|I_0|} (K^{I \backslash w I_0}),
\]
and $d_i = \displaystyle \bigoplus_{\substack{ |I|=i,\\ w \in W_I \backslash C_{I_0}(I) /W(\Lambda)}} {}_{w}\pi_I^{I'} \otimes
\epsilon^{I'}_I$ for $i=|I| \geq |I_0|$; and $C_i=0$, $d_i=0$ otherwise.

Let us identify $M^{W(\Lambda)} = K^{W(\Lambda)} \otimes_{K} M$ with $E_G(\Lambda) \otimes_{K} M$ by the
linear map:
\begin{equation}
  \label{eq:phi}
  \varphi: M^{W(\Lambda)} \rightarrow E_G(\Lambda) \otimes_{K} M \quad \textrm{ defined by }\varphi ((m_w)_{w \in W(\Lambda)}) = \sum_{w \in W(\Lambda)}T_w \otimes \pi (T_w)^{-1} m_w,
\end{equation}
and forget about the $E_G(\Lambda)-$module structure on $K^{W(\Lambda)} \otimes M$.

We use the same argument as Lemma \ref{lem:Kato.Lem1} to obtain for $s \in S(\Lambda)_{w, I}$,
\[
  K^{W(\Lambda)}_s = \{ (x_{w})_{w \in W(\Lambda)} \in K^{W(\Lambda)} \ | \ x_{ws} = -x_w, \ w \in
  W(\Lambda) \},
\]
we have $\varphi^{-1} (L_s) = K^{W(\Lambda)}_s \otimes_K M$.
Now the following holds:
\[
(K^{W(\Lambda)}/ \sum_{s \in S(\Lambda)_{w, I}}K_s^{W(\Lambda)}) \otimes M \cong E_G(\Lambda) \otimes_K M/
\sum_{s \in S(\Lambda)_{w, I}} L_s.
\]
Using the results we get so far, we can write the complex as

\begin{equation}
  \label{eq:Katoeq1.2againx2}
  \begin{aligned}
    0 \rightarrow   \ldots & 0 \xrightarrow{d_{|I_0|-1}}    \bigoplus_{\substack{ |I|=|I_0|,\\ w \in W_I \backslash C_{I_0}(I) /W(\Lambda)}} (K^{W(\Lambda)} / \sum_{s \in S(\Lambda)_{w, I}}K_s^{W(\Lambda)}) \otimes \bigwedge^{|I|-|I_0|} (K^{I \backslash w I_0}) \xrightarrow{d_{|I_0|}} \\
    &  \xrightarrow{d_{|I_0|}} \bigoplus_{\substack{ |I|=|I_0|+1,\\ w \in W_I \backslash C_{I_0}(I) /W(\Lambda)}} (K^{W(\Lambda)} / \sum_{s \in S(\Lambda)_{w, I}}K_s^{W(\Lambda)}) \otimes \bigwedge^{|I|-|I_0|} (K^{I \backslash w I_0}) \xrightarrow{d_{|I_0|+1}} \ldots \\
   & \ldots  \xrightarrow{d_{|S|-1}} C_{|S|} \rightarrow 0
  \end{aligned}
\end{equation}
where in the second line of \refeq{eq:Katoeq1.2againx2}, the $s$ is the unique
element of $S(\Lambda)_{w, I}$.

We can further identify $K^{W(\Lambda)} / \displaystyle \sum_{s \in S(\Lambda)_{w,
    I}}K_s^{W(\Lambda)}$ with $K[W (\Lambda) /W_{\Delta (\Lambda)_{w, I}}]$ by the map
\[
  (x_w)_{w \in W(\Lambda)} \mapsto \sum_{w \in W(\Lambda)}x_w \cdot w
  W_{ \Delta (\Lambda)_{w, I}}.
\]
Then uses the complex
by \cite{Solomon1966} for simplicial decomposition of $(|S|-|I_0|-1)$-dimensional
sphere, or see \cite[the main theorem ]{DeligneLusztig1982}, we have
\[
  D[M] =
\sum_i (-1)^i [C_i] = (-1)^{|I_0|} [\Ker d_{|I_0|}]= (-1)^{|I_0|}[\bigcap_{s \in
S( \Lambda)}L_s],
\]
By Lemma \ref{lem:Kato.Lem1} and the same argument as \refeq{eq:phitochi},
we know that
\[
  \bigcap_{s \in  S(\Lambda)} L_s = (\operatorname{id} \otimes \pi)( \chi) (1
  \otimes M) ,
\]
where \[
  \chi_{I_0^{\bot}} = \sum_{w \in W(\Lambda)} (-1)^{\ell_{I_0^{\perp}}(w)}T_w
  \otimes T_w^{-1}.
\]
We now need to show that $(-1)^{|I_0|} \chi_{I_0^{\bot}} (1 \otimes M)$ is isomorphic to $M^{*}$, or
 equivalently:
 \[
   (T_{w } \otimes  1) \chi_{I_0^{\bot}} = \chi_{I_0^{\bot}} (1 \otimes T_w^{*}) = \chi_{I_0^{\bot}} (1 \otimes (-1)^{\ell_{I_0^{\bot}}(
     w )} p_w T_{w^{-1}}^{-1}).
 \]
Hence we need to show $\chi$ intertwines $E_G(\Lambda)$-action and twisted
$E_G(\Lambda)$-action as above, but this is done using Lemma \ref{lem:Kato.Lem2} for the finite
case.

\end{proof}

\section{Representation theory of $p$-adic groups via Hecke algebras}\label{sec:Reppadic}

\subsection{Notation and preliminaries}\label{subsec:padicgrppreliminaries}

In this section, we introduce notations for representation of $p$-adic groups. Let $\mathbb{G}$ be a split group over a non archimedean local field $F$, $G=\mathbb{G} (F)$ denote its $F$-rational points.

Let $H \subset G$ be a subgroup and
$(\rho, V_{\rho})$ a smooth representation of $H$ (\emph{i.e.} for any $v \in V$
, there is a compact open subgroup $K$ fixing $v$), we write $(\ind_{H}^G(\rho),
\ind_H^G (V_{\rho}))$ for the
compactly induced representation. For $ g \in G$, ${}^g (\textrm{-})$ denote the action
$g (\textrm{-} )g^{-1}$ on $G$, ${}^g \rho$ denote the representation $\rho
({}^{g^{-1}} (\textrm{-}))$. From now on until the
end of this paragraph, we assume
$H$ is open in $G$. For a smooth $(\pi,V)$ representation of $G$,
we write $(\res^G_H (\pi), \res^G_H(V))$ for the restriction functor. We have an
adjoint pair $(\ind_H^G, \res_H^G)$ in the sense that there exists a natural
equivalence $\Hom_G (\ind_H^G (\rho), -) \cong \Hom_H (\rho ,  \res^G_H (-))$.

For any smooth representation $(\rho, V_{\rho})$ let $(\rho^{\vee},
V_{\rho}^{\vee})$ denote the smooth contragredient of $(\rho, V_{\rho})$.

Let us fix a minimal $F$-subgroup $P_0$ of $G$ and a maximal $F$-split
torus $T$ contained in $P_0$. A parabolic subgroup $P$ of $G$, with Levi
subgroup $L$ is said to be \emph{standard} if $P \supset P_0$ and $L \supset T$.
We have explicitly $P = LU$ where $U$ is the unipotent radical and we also write
the opposite Levi by $\overline{P} = L \overline{U}$. Let $\mathcal{P}(L)$
denote the set of parabolic subgroups with Levi component $L$. Let $W (G,T):=N_G (T)/T$
denote the (finite)
Weyl group of $G$, $R$ the set of roots and $S$ the set of simple roots
determined by the choice of $P_0$. For $I \subset S$, let $P_I$ denote the
standard parabolic $F$-subgroup of $G$ determined by $I$, and $L_I$ the Levi
subgroup of $P_I$. We write
$\widehat{T}$ for the complex torus dual to $T$, $\widehat{G}$ for the Langlands
dual group of $G$, $W_F$ for the Weil group of
$F$, $I_F$ for the inertial group of $F$.

For a smooth  $(\rho, V_{\rho})$ representation of a Levi subgroup $L \subset P=LU$, $\rho$ extends to a
representation of $P$ by letting $U$ act trivially. Let $(\Ind_P^G (\rho),
\Ind_P^G (V_{\rho}))$ denote the un-normalized induction. For a smooth
representation $(\pi, V)$ of $G$, let $(\pi_U, V_U)$ denote the un-normalized
Jaquet module of $(\pi, V)$ and $j_U: V \rightarrow V_U$ denote the quotient map. We denote the normalized
induction and normalized restriction(Jacquet module) of $(\rho, V_{\rho})$ by
$(\NI_P^G (\rho), \NI_P^G (V_{\rho}) )$ and $(\Nr_P^G (\rho), \Nr_P^G
(V_{\rho}))$ (or $(\Nr_U(\rho), \Nr_U(V_{\rho}))$) respectively. We have two
adjoint pairs: $(r_U, \NI_P^G)$, where there exists a natural equivalence $\Hom_G ( - , \NI_P^G (\rho)) \cong \Hom_L (r_U (-),
\rho)$ and $(\NI_P^G, r_{\overline{U}})$ with  $\Hom_G ( \NI_P^G (\rho) , - ) \cong \Hom_L (\rho,
 r_{\overline{U}}(-))$ (the second adjoint theorem, see \cite[VI.9]{Renard2010}).

Let $M$ be a Levi subgroup of $G$, $\Rep (M)$ denote the category of all smooth
complex representations of $M$,  $\Repf (M)$ denote the subcategory of finite
length representations, and $\operatorname{Irr}(M)$ denote the set of isomorphism classes of smooth
irreducible representations. Let $\mathfrak{R}(G)$ be the Grothendieck group of
$\Rep(G)$ and for $\pi \in \Rep(G)$, denote by $[\pi]$ its image in $\mathfrak{R}(G)$.

We now introduce the \emph{Aubert-Zelevinsky duality}:
  For $\pi \in \Repf (G)$ we follow \cite{Aubert1995}, consider the map
\begin{equation}\label{eq:AZDuality}
  \begin{aligned}
    \ND_G: \mathfrak{R}(G) & \rightarrow  \mathfrak{R}(G)  \\
                      [\pi]   &  \mapsto [\sum_{I \subset S} (-1)^{|I|} \NI_{P_I}^{G} \circ \Nr^G_{P_I} (\pi) ].
  \end{aligned}
\end{equation}

\subsection{The Bernstein decomposition and an equivalence of categories}
Let $\pi \in
\operatorname{Irr}(G)$, there exists a parabolic subgroup $P=L U$ of $G$ and a
supercuspidal irreducible representation $\sigma$ of $L$ such that $\pi$ embeds
in $\NI_P^G \sigma$. The pair $(L, \sigma)$ is unique up to $G$-conjugacy. The
$G$-conjugacy class $(L, \sigma)_G$ of $(L, \sigma)$ is called the \emph{supercuspidal
support} of $\pi$. We denote by $\operatorname{Sc}$ the map defined by

\begin{equation}
  \begin{aligned}
\operatorname{Irr}(G) & \rightarrow \left\{ \textrm{ equivalence classes of supercuspidal pairs} \right\}  \\
  \pi & \mapsto (L, \sigma) \textrm{ if } \pi \cong \textrm{ a }G\textrm{-subquotient of }\NI_{P}^{G}  (\sigma) .
  \end{aligned}
\end{equation}

Let $L$ be a Levi subgroup of a parabolic subgroup $P$ of $G$, and let
$\mathfrak{X}_{\mathrm{nr}}(L)$ denote the group of unramified characters of
$L$. Given two cuspidal pairs $\left(L_i, \sigma_i\right)(i=1,2)$, they are
\emph{inertially equivalent} if there exists a $g \in G$ and $\nu \in
\mathfrak{X}_{\mathrm{nr}} (L_2)$ such that:
$$
L_2={ }^g L_1 \quad \text { and } \quad{ }^g \sigma_1 \simeq \sigma_2 \otimes
\nu .
$$

We write $[L, \sigma]_{G}$ for the inertial equivalence class of $(L, \sigma)$
(sometimes omitting the sub-index if no confusion is caused) and
$\mathfrak{B}(G)$ for the set of all inertial equivalence classes. If $M$ is a
proper Levi subgroup of $G$, $L$ is a Levi subgroup of $M$ and $\sigma$ is a
supercuspidal representation of $L$, then such $\mathfrak{s}_M = [L, \sigma]_M \in
\mathfrak{B}(M)$ determines a $\mathfrak{s}_G=[L,\sigma]_G \in
\mathfrak{B}(G)$ naturally. If $(\pi ,
V)$ is an irreducible smooth representation of $G$, one defines the \emph{inertial support}
$\mathfrak{I}\left(\pi \right)$ of $(\pi, V)$ to be the inertial equivalence class of the
supercuspidal support of $\pi$.

Now fix a class $\mathfrak{s} \in \mathfrak{B}(G)$. We denote by $
\Rep^{\mathfrak{s}}(G)$ the full subcategory of $\Rep (G)$ defined as follows:

Let $(\pi, \mathcal{V}) \in  \Rep(G)$. Then $(\pi, \mathcal{V}) \in \Rep^{\mathfrak{s}}(G)$ if and only if every irreducible $G$-subquotient $\pi_0$ of $\pi$ satisfies $\mathfrak{I}\left(\pi_0\right)=\mathfrak{s}$.

The subcategories $ \Rep^{\mathfrak{s}}(G) , \ \mathfrak{s} \in \mathfrak{B}(G)$
split the category $\Rep (G)$. This is the Bernstein decomposition of
$\Rep (G)$ (see \cite{Bernstein1984}):
\begin{equation}
  \label{eq:BernsteinDecomposition}
  \Rep (G)=\prod_{\mathfrak{s} \in \mathfrak{B}(G)} \Rep^{\mathfrak{s}}(G)
\end{equation}
of the subcategories $\Rep^{\mathfrak{s}}(G)$ as $\mathfrak{s}$ ranges through
$\mathfrak{B}(G)$. Let
\begin{equation}
  \label{eq:prpadic}
  \operatorname{pr}_G^{\mathfrak{s}}: \Rep (G) \rightarrow
  \Rep^{\mathfrak{s}}(G)
\end{equation}
denote the projection functor.

For $\mathfrak{s}=[L, \sigma]_G \in \mathfrak{B}(G)$, following
\cite[section 3.2]{AubertXu2023explicit} we define $
\mathrm{N}_G\left(\mathfrak{s}_L\right):=\left\{g \in G:{ }^g L=L\right.$ and ${
}^g \sigma \simeq \chi \otimes \sigma$, for some $\left.\chi \in
  \mathfrak{X}_{\mathrm{nr}}(L)\right\}$, where $\mathfrak{s}_L=[L, \sigma]_L
\in \mathfrak{B}(L)$ and denote by $W_G^{\mathfrak{s}}$ the extended finite Weyl
group $\mathrm{N}_G\left(\mathfrak{s}_L\right) / L$ and $\Waff^{\mathfrak{s}}$
the corresponding affine version.

Let $L$ be a fixed Levi subgroup of $G$, $\sigma$ be a fixed supercuspidal
representation of $L$ and $\mathfrak{s}= [L, \sigma]_G$. We recall Bernstein's construction of a progenerator of
$\Rep^{\mathfrak{s}}(G)$ following \cite[Section
1]{Roche2002}. Let $L^1= \bigcap_{\nu \in \mathfrak{X}_{\mathrm{nr}} (L)} \Ker
\nu$, there exists an irreducible supercuspidal subrepresentation $\sigma_1$ of
$\sigma |_{L^1}=\sigma_1 \oplus \sigma_2 \oplus \ldots \oplus \sigma_r$. Let
$\Sigma$ denote $\ind_{L^1}^L \sigma_1$, we know it is a progenerator of the
category $\Rep^{\mathfrak{s}}(L)$. Now consider $\NI_P^G (\Sigma)$ for some $P
\in \mathcal{P} (L)$. We have (see \cite[Section 1.4-1.6]{Roche2002})
\begin{thm}\label{thm:Roche}
  The isomorphism class of $\NI_P^G (\Sigma)$ is independent of the parabolic
  subgroup $P \in \mathcal{P}(L)$.
\end{thm}

Combining with the fact that $\bigoplus_{P \in \mathcal{P}(L)} \NI_P^G (\Sigma)$
is a progenerator, we have an equivalence of categories:
\begin{cor}\label{cor:Roche1.6}
 For any $P \in \mathcal{P}(L)$, the representation $\NI_P^G (\Sigma)$ is a
 progenerator of $\Rep^{\mathfrak{s}} (G)$,
 \begin{equation}
   \begin{aligned}
    \mathcal{E}_G: \Rep^{\mathfrak{s}} (G) \rightarrow & \Mod \textrm{-} \End_G (\NI_P^G \Sigma) \\
   V \mapsto & \Hom_G (\NI_P^G (\Sigma), V)
   \end{aligned}
 \end{equation}
is an equivalence of categories.
\end{cor}
We may omit the sub-index of $\mathcal{E}_{G}$ if no confusion is caused.
\subsection{Induction and restriction}
Let $M$ be a Levi subgroup such that $L \subset M$, $P = LU \in \mathcal{P} (L)$
and $Q =MN \in \mathcal{P}(M)$. The pair $(L, \sigma)$ determines
$\mathfrak{s}_M = [L, \sigma]_M$ and $\mathfrak{s}= [L, \sigma]_G$.
\begin{thm}[Theorem of section 5, \cite{Roche2002}]
  The two following diagrams commute up to natural equivalence
  \begin{center}
    \begin{equation}
  \label{eq:RocheRUop}
  \xymatrix@C=4.5em@R=2.5em{
    \Rep^{\mathfrak{s}}(G) \ar[d]_{\Nr_{\overline{U}}^{\mathfrak{s}_M}} \ar[r]^{\mathcal{E}_G} & \Modb \End_G (\NI_P^G \Sigma) \ar[d]^{\Res^{\small \End_G (\NI_P^G \Sigma)}_{ \small \End_M (\NI_{M \cap P}^{M} \Sigma)}}\\
    \Rep^{\mathfrak{s}_M} (M)  \ar[r]_{\mathcal{E}_M} & \Modb \End_M (\NI_{M \cap P}^{M} \Sigma)
  }
\end{equation}

\begin{equation}
  \label{eq:RocheInd}
  \xymatrix@C=4.5em@R=2.5em{
    \Rep^{\mathfrak{s}}(G) \ar[r]^{\mathcal{E}_G} & \Modb \End_G (\NI_P^G \Sigma)\\
    \Rep^{\mathfrak{s}_M}(M) \ar[u]^{\NI_Q^G} \ar[r]_{\mathcal{E}_{M}} & \Modb \End_M (\NI_{M \cap P}^{M} \Sigma) \ar[u]_{\Ind_{\small \End_M (\NI_{M \cap P}^{M} \Sigma)}^{\small \End_G (\NI_P^G \Sigma)}}
  }
\end{equation}
    \end{center}
    In the first diagram \refeq{eq:RocheRUop}, $\Nr_{\bar{U}}^{\mathfrak{s}_M} :=\pr_{\mathfrak{s}_M}
    \circ \Nr^G_{L \bar{U}}$, $\Res^{\small \End_G (\NI_P^G \Sigma)}_{
      \small \End_M (\NI_{M \cap P}^{M} \Sigma)}$ is the restriction along the
    natural ring injection $\End_M (\NI_{M \cap P}^{M} \Sigma) \rightarrow  \End_G (\NI_Q^G
    (\NI_{M \cap P}^{M} \Sigma)) = \End_G (\NI_P^G \Sigma)$ provided by the
    functor $\NI_Q^G$. In the second diagram \refeq{eq:RocheInd},
    $\Ind_{\small \End_M (\NI_{M \cap P}^{M} \Sigma)}^{\small \End_G (\NI_P^G
      \Sigma)} = \textrm{-} \bigotimes_{\small \End_M (\NI_{M \cap P}^{M}
      \Sigma)}\small \End_G (\NI_P^G \Sigma)$.
  \end{thm}
We now apply this theorem to the opposite parabolic subgroups $\overline{P}$ and
$\overline{Q}$, noticing the fact $ \NI_P^G \Sigma \cong \NI_{\overline{P}}^G \Sigma$ and $ \NI_{P \cap M}^M \Sigma \cong \NI_{\overline{P} \cap
  M}^{M} \Sigma$  from the proof of \ref{thm:Roche}, we have
\begin{center}
    \begin{equation}
  \label{eq:RocheRUop1}
  \xymatrix@C=4.5em@R=2.5em{
    \Rep^{\mathfrak{s}}(G)  \ar[d]_{\Nr_{U}^{\mathfrak{s}_M}} \ar[r]^{\mathcal{E}_{G}} & \Modb \End_G (\NI_P^G \Sigma) \ar[d]^{\Res^{\small \End_G (\NI_P^G \Sigma)}_{ \small \End_M (\NI_{M \cap P}^{M} \Sigma)}}\\
    \Rep^{\mathfrak{s}_M}(M)  \ar[r]_{\mathcal{E}_M}& \Modb \End_M (\NI_{M \cap P}^{M} \Sigma)
  }
\end{equation}
  \end{center}
  \begin{center}
    \begin{equation}
  \label{eq:RocheInd1}
  \xymatrix@C=4.5em@R=2.5em{
    \Rep^{\mathfrak{s}}(G) \ar[r]^{\mathcal{E}_G} & \Modb \End_G (\NI_P^G \Sigma)\\
    \Rep^{\mathfrak{s}_M}(M) \ar[u]^{\NI_{\overline{Q}}^G} \ar[r]_{\mathcal{E}_{M}} & \Modb \End_M (\NI_{M \cap P}^{M} \Sigma) \ar[u]_{\Ind_{\small \End_M (\NI_{M \cap P}^{M} \Sigma)}^{\small \End_G (\NI_P^G \Sigma)}}
  }
\end{equation}
 \end{center}
    Comparing \refeq{eq:RocheInd} and \refeq{eq:RocheInd1}, we find
    \begin{cor}
      Under the assumptions of this section, $\NI_{\overline{Q}}^G $ is
      equivalent to $\NI_Q^{G}$.
    \end{cor}

\section{Comparing the involution with the Aubert-Zelevinsky duality}\label{subsec:comparison}
Let us fix a $\mathfrak{s} = [L, \sigma]_G \in \mathfrak{B}(G)$ with $L$ a
standard Levi subgroup, and $P = LU \in
\mathcal{P}(L)$.
Aubert-Zelevinsky duality $\ND_G =\sum_{I \subset S} (-1)^{|I|} \NI_{P_I}^{G} \circ \Nr^G_{P_I}$ acts on $[\pi]$
for $\pi \in \Rep^{\mathfrak{s}}_{\operatorname{f}}(G)$ by the formula
\refeq{eq:AZDuality}. Let $\pi_0$ denote an irreducible component of $\pi$ that
is a subquotient of $\NI_{P}^G (\sigma \otimes \nu)$ for some $\nu \in \mathfrak{X}_{\mathrm{nr}}(L)$, then $\Nr^G_{P_I} (\pi_0)$ is a
subquotient of $\Nr_{P_I}^G (\NI_{P}^G (\sigma \otimes \nu))$. From the
\cite[Geometric Lemma 2.11]{BernsteinZelevinsky1977}, $ \Nr_{P_I}^G \circ \NI_{P}^G $ admits a
filtration by functors of the form $\NI^{L_I}_{w(L) \cap L_I} \circ w \circ
\Nr^L_{L \cap w^{-1} (L_I)} $ for $w \in
W^{L,L_I}$ (see the definition of $W^{L, L_I}$ in
\emph{loc.cit.} 2.11). Then $\Nr_{P_I}^G (\pi_0)$ is contained in subquotients
of the sum of $\NI^{L_I}_{w(L) } (w (\sigma \otimes \nu) )$ for all  $w \in W^{L, L_I}$ satisfying
$w(L) \subset L_I$. The pair $(w(L), w (\sigma \otimes \nu))$
determines $[w(L), w (\sigma \otimes \nu)]_{L_I} \in \mathfrak{B}(L_I)$ (in general
differs from $[L, \sigma ]_{L_I}$), by
fixing a representative $g \in G$ of $w$, we see such pair corresponds
to $\mathfrak{s}$ in $\mathfrak{B}(G)$: $[{}^{g^{-1}} (w(L)), {}^{g^{-1}} (w (\sigma
\otimes \nu)) ]_G =[L,\sigma ]_G \in \mathfrak{B}(G)$.

We denote by $\mathfrak{S}_I  := \{ \mathfrak{s}_{L_I} =[w(L), w(\sigma)]_{L_I} \
| \  w \in W^{L, L_I},\ w (L) \subset L_I \}$.  We now adapt the diagrams
\refeq{eq:RocheRUop1} and \refeq{eq:RocheInd1} accordingly. Let $M = L_I$,
$\mathfrak{s}_{L_I} \in \mathfrak{S}_I $, the progenerator of
$\Rep^{[w(L),w(\sigma)]_L}(w(L))$ is $\Sigma_w = \ind_{w (L^1)}^{w(L)} (w
(\sigma_1))$. We also have $\End_G (\NI_P^G (w(\Sigma))) \cong \End_G
(\NI_P^G (\Sigma) )$ by \ref{thm:Roche}, the diagrams become

\begin{center}
   \begin{equation}
\label{eq:RocheRUop2}
\xymatrix@C=4.5em@R=2.5em{
{\Rep^{\mathfrak{s}}(G)} \ar[d]_{\operatorname{pr}_{L_I}^{\mathfrak{s}_{L_I}} \circ \Nr_{P_I}^{G}} \ar[r]^{\mathcal{E}_G} & {\Modb \End_G (\NI_P^G \Sigma)} \ar[d]^{\Res^{\small \End_G (\NI_P^G \Sigma)}_{\small{\End_{L_I} (\NI_{L_I \cap P}^{L_I} \Sigma_w)}}} \\
\Rep^{\mathfrak{s}_{L_I}}(L_I) \ar[r]_{\mathcal{E}_{L_I}} & \Modb \End_{L_I} (\NI_{L_I \cap P}^{L_I}  \Sigma_w)
}
\end{equation}

\begin{equation}
\label{eq:RocheInd2}
\xymatrix@C=4.5em@R=2.5em{
{\Rep^{\mathfrak{s}}(G)} \ar[r]^{\mathcal{E}_G} & {\Modb \End_G (\NI_P^G  \Sigma)}  \\
\Rep^{\mathfrak{s}_{L_I}}(L_I) \ar[r]_{\mathcal{E}_{L_I}} \ar[u]^{\NI_{P_I}^G}  & \Modb \End_{L_I} (\NI_{L_I \cap P}^{L_I}  \Sigma_w) \ar[u]_{\Ind_{\small \End_{L_I} (\NI_{L_I \cap P}^{L_I} \Sigma_w)}^{\small \End_G (\NI_P^G \Sigma)}}
}
\end{equation}
\end{center}
where $\pr_{L_I}$ is define in \refeq{eq:prpadic}. We have shown in the above
discussion that the Bernstein
blocks $\Rep^{\mathfrak{s}_{L_I}}(L_I)$ with $\pr_{L_I}^{\mathfrak{s}_{L_I}}
\circ \Nr_{P_I}^G (\pi_0) \neq 0$ are all indexed by some $\mathfrak{s}_{L_I}
\in \mathfrak{S}_I$, thus the functor $\Nr_{P_I}^G$ restricted to $\Reps (G)$
can be written as a sum $\sum_{\mathfrak{s}_{L_I} \in \mathfrak{S}_I} \operatorname{pr}_{L_I}^{\mathfrak{s}_{L_I}}
\circ \Nr_{P_I}^G$. The duality functor has the form
\[
  \ND_G =\sum_{I \subset S}
(-1)^{|I|} (\sum_{\mathfrak{s}_{L_I} \in \mathfrak{S}_I} \NI_{P_I}^{G} \circ  \operatorname{pr}^{\mathfrak{s}_{L_I}}
\circ \Nr_{P_I}^G).
\]
By \cite{Solleveld2023endomorphism}, we know all these
endomorphism algebras are generalized (in the sense twisted with $2$-cocycle)
affine Hecke algebras. On the Hecke algebra side, the involution is
\begin{center}
  \begin{equation}
    \label{eq:DH}
    \ND_{\H} =\sum_{I
      \subset S} (-1)^{|I|} (\sum_{\mathfrak{s}_{L_I} \in \mathfrak{S}_I} \Ind_{\small \End_{L_I} (\NI_{L_I \cap P}^{L_I}
      \Sigma_w)}^{\small \End_G (\NI_P^G \Sigma)} \circ \Res^{\small \End_G (\NI_P^G \Sigma)}_{ \small{\End_{L_I} (\NI_{L_I \cap P}^{L_I} \Sigma_w)}}).
  \end{equation}
\end{center}

We  see the following facts:
\begin{enumerate}
\item If the supercuspidal support is $(T,1)$,
  in this case $\mathfrak{S}_I =\{ \mathfrak{s}_{L_I}= [T,1]_{L_I} \}$ is a
  singleton, $\Nr_{P_I}^G: \Reps (G) \rightarrow \Rep^{\mathfrak{s}_{L_I}} (L_I)$
  is the restriction of $\Nr_{P_I}^G$ to $\Reps (G)$. The involution on the Hecke
  algebra side becomes  \refeq{eq:InvolutionH} in the Section \ref{chp:involution}
  after passing to the category of left modules of the opposite algebras.
\item  In the finite group of Lie type cases, we interpreted the left hand side of
  \refeq{eq:thm:HLanalogue} as the counterpart on the ``finite Hecke algebra side'' of
  Alvis-Curtis-Kawanaka duality. The underlying ideas inspire the approach that we
  take in this Section to get \refeq{eq:DH} which is introduced as the counterpart on the ``generalized affine Hecke algebra side'' of
 Aubert-Zelevinsky duality.
\end{enumerate}

\appendix
\begin{appendices}
\section{Homology Representations}\label{app:A}

This Appendix is devoted to topology background and the construction of homology
representations that we need for Section \ref{subsec:HLanalogue}.
\subsection{The topology vocabularies}\label{subsec:TopVocabularies}

  An (abstract) finite \emph{simplicial complex} $\Sigma$ consists of a finite set
of points $V(\Sigma):= \{
x_0,x_1, \ldots   x_{N}\}$ (for some $N \in \mathbb{N}$) called \emph{vertices}, together with certain finite,
non-empty sets of vertices $\left\{ \sigma \right\}$ called \emph{simplices},
satisfying the axioms that
\begin{enumerate}[(i)]
\item each singleton set $\left\{ x \right\}$ is a simplex,
\item each non-empty subset $\sigma'$ of a simplex $\sigma$ a is also a simplex.
\end{enumerate}
\emph{A simplicial map} $f: \Sigma \rightarrow \Sigma'$ of simplicial complexes is a map $f$ from the vertices of
$\Sigma$ to those of $\Sigma'$, such that if $\sigma = \left\{ x_0, x_1, x_2,
  \ldots , x_r \right\}$ is an $r$-simplex (\emph{i.e.} simplex consisting of
$r+1$ vertices) in $\Sigma$, then $\{ f(x_0 ), f(x_1),
... ,f(x_r) \}$ are the vertices (possibly with repetitions)
of a simplex in $\Sigma'$. The finite simplicial complexes form a category, in which
morphisms are simplicial maps.
We call a subset $A$ of a simplex $\sigma$, a \emph{face} of $\sigma$. A \emph{subcomplex} of
$\Sigma$ is a collection of subsets in $\Sigma$ which is a simplicial complex in
its own right, with the vertex set being some subset of $V(\Sigma)$.

We give a more useful definition of simplicial complex via posets:
\begin{defn}[Simplicial complex]\label{defn:SimplicialComplex}
  Any poset $\Sigma$ satisfying the properties below is referred to as a
  \emph{simplicial complex}:
  \begin{enumerate}[(a)]
  \item Any two elements $A$, $B \in \Sigma$, have a greatest lower bound,
    denoted by $A \cap B$.
    \item For $A \in \Sigma$, the poset of all faces $\Sigma_{\preceq A}$ is
      isomorphic to $2^{ \left\{ 1,2, \ldots, r \right\}}$ for some positive
      integer $r$ with inclusion ordering.
  \end{enumerate}
\end{defn}

 Now let $G$ be a finite group, a simplicial complex $\Sigma$ is called a
\emph{$G$-complex} and $G$ is said to act on $\Sigma$ if the vertices of
$\Sigma$ form a $G$-set and if the action of $G$ carries simplices to simplices. The finite $G$-complexes form a category, in which the objects are $G$-complexes, and the morphisms are simplicial maps preserving the $G$-action.

\begin{defn}
  \label{defn:TopSigma}
To each simplicial complex $\Sigma$ we associate an underlying topological
space, denoted by $|\Sigma|$, consisting of all real valued functions $p$ such
that:
\begin{enumerate}[(i)]
\item $p(x) \geq 0$ for all vertices in $\Sigma$,
\item $\sum_{x \in V(\Sigma)} p(x) =1$,
  \item $\supp p(x) :=\{ x \in V(\Sigma) \ | \ p(x) \neq 0 \}$ is a simplex of $\Sigma$.
  \end{enumerate}
\end{defn}
We now describe the topology structure of $|\Sigma |$. For each $n$-simplex $\sigma = \left\{ x_{0}, x_1, \ldots , x_n \right\}$ of $\Sigma$, denote by $|\sigma |$ the subset of
$|\Sigma |$ defined by
\[
|\sigma| = \{ p \in |\Sigma | \ | \ \supp p \subseteq \sigma \}.
\]
Each point $p \in |\sigma| $ determines a point $\sum_{i=1}^n p(x_i) e_i$ in
$\mathbb{R}^n$, where $\left\{ e_i \right\}_{i=1}^n$ is an orthonormal basis of
$\mathbb{R}^n$. In view of the condition $(ii)$ above, we give a bijection
between $|\sigma|$ and the standard $n$-simplex $|\left\{ 0, e_1, e_2, \ldots,
  e_n \right\}|$ of $\mathbb{R}^n$, and topologize $| \sigma |$ via this
bijection. Finally we topologize $\Sigma$ by the family of subsets $U$ such that
$U \cap |\sigma |$ is open in $|\sigma |$ for every simplex $\sigma $ of
$\Sigma$. A simplicial map $f: \Sigma \rightarrow \Sigma'$ yields a continuous
map $|f| : |\Sigma| \rightarrow |\Sigma' |$, defined by
\[
p = \sum_{x \in V(\Sigma)} p(x)x \mapsto |f|p = \sum_{x \in V(\Sigma)} p(x) f(x)
\in |\Sigma'|.
\]

We can talk about things like closure, dimension of a simplex using the above topological
description, for example, a $r$-simplex is of dimension $r$, we call a simplex
of maximal dimension a \emph{chamber}.

A finite group $G$ is said to act on a chain complex $C_{\bullet}= (C_{r}, \partial_r)_{r
\geq 0}$ ($\partial_r C_r \subseteq C_{r-1}$, $\partial_0 C_0 =0$ and $\partial^2=0$) if each subspace $C_r$
is a finite dimensional $K G$-module and if $\partial$ commutes with the action
of $G$. We refer to $C_{\bullet}$ as a chain complex with $G$-action in this case, and
then the homology spaces $\left\{ H_r(C_{\bullet})\right\}$ are $K G$-modules, affording what we shall call \emph{homology representations} of $G$.

Now let $\Sigma$ be a $G$-complex, and $K$ a field. We now define a chain
complex $C_{\bullet}(\Sigma)$ with $G$-action. For $r \geq 0$, the subspace
$C_r(\Sigma)$ of $r$-chains has a $K$-basis $\left\{c_\sigma\right\}$, indexed
by a $r$-simplices $\sigma$ in $\Sigma$. The action of $G$ is given by its permutation action on the $r$-chains:
$$
g c_\sigma=c_{g \sigma}, \quad \textrm{ for } \sigma \text { as above, and } g \in G \text {. }
$$
The boundary homomorphism $\partial_r: C_r(\Sigma) \rightarrow C_{r-1}(\Sigma)$ is given by
$$
\partial c_\sigma=\sum_{i=0}^r(-1)^i c_{\sigma_i},
$$
where if $\sigma= \{ x_0,\cdots , x_r\}$, then
$\sigma_i=\{ x_0, \cdots, \hat{x}_i, \cdots , x_r\}$ is the $i$-th face of
$\sigma$, for $0 \leq i \leq r$ with $x_i$ omitted. The vector space $C_r(\Sigma)$ is
called the \emph{vector space of $r$-chains}, where an $r$-chain is a formal
linear combination of $r$-simplices with coefficients in $K$. We can verify that
$\partial^2=0$, hence $C_{\bullet} (\Sigma)$ is a chain complex with $G$-action, and the resulting homology representation
\[
H_*(C_{\bullet} (\Sigma))=\bigoplus_{r=0}^{\infty} H_r(C_{\bullet} (\Sigma))
\]
is called \emph{the homology representation} associated with the $G$-complex $\Sigma$.

\begin{defn}\label{defn:LefschetzChar}
  Let $\Sigma$ be a $G$-complex.  The Lefschetz character of the homology representation
  $H_*(C_{\bullet} (\Sigma))$ is the map $\Lef: G \rightarrow K$ defined by
\[
\Lef (g)=\sum_{r \geq 0}(-1)^r  \Tr \left(g, H_r(C_{\bullet} (\Sigma))\right) .
\]
The degree $\Lef (1)$ of the Lefschetz character is called the \emph{Euler characteristic} of $\Sigma$, and is denoted by $\chi(\Sigma)$.
\end{defn}
In the case $K$ is a field of characteristic zero the Euler characteristic
coincides with the alternating sum
\[
\sum_{r=0 }^{\infty} (-1)^r \dim H_r (C_{\bullet} (\Sigma)),
\]
which is the \emph{(usual) Euler characteristic} $\chi (|\Sigma|)$ of the underlying topological
space $|\Sigma|$.

\begin{prop}[Hopf Trace Formula.]\label{prop:Hopf}
 Let $C_{\bullet}=\bigoplus_{r>0} C_r$ be a chain complex over a field $K$, with boundary
 map $\partial$, such that each subspace $C_r$ is finite dimensional over $K$,
 and $C_r=0$ for all sufficiently large $r$. Let $f: C_{\bullet} \rightarrow
 C_{\bullet} $ be a \emph{graded chain map of degree zero}. (\emph{i.e.} a $K$-endomorphism $f: C_{\bullet} \rightarrow C_{\bullet}$ such that $f\left(C_r\right) \subseteq C_r$ for each $r \geq 0$, and $f \partial=\partial f$.) Then $f$ induces a $K$-endomorphism $f_*$ of $H_*(C_{\bullet})$ such that $f_*\left(H_r(C_{\bullet})\right) \subseteq H_r(C_{\bullet})$ for all $r$, and we have
$$
\sum_{r=0}^{\infty}(-1)^r \operatorname{Tr}\left(f, C_r\right)=\sum_{r=0}^{\infty}(-1)^r \operatorname{Tr}\left(f_*, H_r(C)\right) .
$$
\end{prop}
The $G$-action on $C_{\bullet}$ is a graded chain map of degree zero, hence we have
\begin{cor}
  Let $K$ be a field of characteristic zero. Then the Euler characteristic $\chi(\Sigma)$ of a $G$-poset $\Sigma$ is given by
\[
\chi(\Sigma)=\sum_{r=0}^{\infty}(-1)^r \operatorname{dim} C_r(\Sigma).
\]
\end{cor}
We now have:
\begin{prop}\label{prop:LefschetzComputation}
Let $\Sigma$ be a $G$-complex, and $\Lef$ the Lefschetz character of the homology representation $H_*(C_{\bullet}(\Sigma))$ of $G$. Then we have
$$
\Lef (g)= \chi\left(\Sigma^g\right)=\chi\left(|\Sigma|^g\right), \quad \textrm{ for each } g \in G \text {. }
$$
where $|\Sigma|^g$ is the fixed point set under the action of $g$,
and $\chi (|\Sigma|)$ is the usual Euler characteristic introduced above.
\end{prop}

\subsection{The Coxeter complex}\label{subsec:CoxeterComplex}
Let $R$ be a root system in an $n$-dimensional euclidean space $E= (V, (\cdot,
\cdot))$ ($\dim V =n$), and let $W$
be the finite Weyl group associated with $R$. The inner product $(\cdot, \cdot)$
corresponds to the perfect pairing in Section \ref{subsec:RootSystemAffineWeyl}.
Since $W \subseteq O(E)$, $W$ acts on the unit sphere $S^{n-1}$. Let $\Delta$ be
a fundamental system in $R$, it determines the set of positive roots $R^{+}$ of
$R$ as the elements in $R$ that can be written as $\mathbb{Z}^{\geq 0}$-linear
combinations of elements of $\Delta$.

For each root $\alpha \in R^{+}$, we define
$$
\begin{aligned}
H_{\alpha}^{+} & =\{v \in V \ | \ ( \alpha, v)>0\}, \\
H_{\alpha}^{-} & =\{v \in V \ | \ (\alpha, v)<0\}, \\
H_{\alpha}= H_{\alpha}^0 & =\{v \in V \ | \ (\alpha, v)=0\}.
\end{aligned}
$$
\begin{defn}\label{defn:CoxeterComplex}
The Coxeter complex is the collection of all subsets of $V$ of the form
$$
\underset{\alpha \in R^{+}}{\bigcap} H_\alpha^{\epsilon_{\alpha}}, \quad \epsilon_{\alpha}=+,- \textrm{, or } 0.
$$
The simplices of the Coxeter complex contained in the closure of the fundamental
chamber (\emph{i.e.} $C:= \underset{\alpha \in R^{+}}{\bigcap} H_\alpha^{+}$) are those of the form
$$
C_I=\left\{v \in V \ | \ \begin{array}{l}
(v, \alpha)=0 \text { for } \alpha \in I \\
(v, \alpha)>0 \text { for } \alpha \in \Delta - I
\end{array}\right\},
$$
where $I$ is any subset of $\Delta$.
\end{defn}
We have an operation of the Weyl group $W$ on the Coxeter complex. From
\cite[Proposition 2.6.1., Proposition 2.6.2. and Proposition 2.6.3.]{Carter1972}, we sumarize:
\begin{prop}\label{prop:CoxeterComplexProperties}
  \begin{enumerate}[(i)]
  \item For any subset $I$ of $\Delta$, the stabilizer of $C_I$ in $W$ is the standard parabolic subgroup $W_I$
 associated with $I$. It is also the point-wise stabilizer of all vectors in
 $C_I$.
 \item Each simplex of the Coxeter complex can be transformed into exactly one
    $C_I$ by an element of the Weyl group.
   \item The parabolic subgroups of $W$ are the stabilizers in $W$ of the
     simplices of the Coxeter complex, \emph{i.e.} the stabilizers of $wC_I$ are
     the parabolic subgroups ${}^wW_I$.
  \end{enumerate}
\end{prop}
Thus we have the following Corollary
\begin{cor}\label{cor:cplxposet}
 The map $w C_I \rightarrow w W_I, w \in W, I \subset \Delta$, is a bijection
 from the set of faces of the closed chambers $\{w \overline{C} \}_{w \in W}$ to the set of left cosets of the parabolic subgroup $W_I$. Moreover,
$$
w C_I \subseteq w^{\prime} C_{I^{\prime}} \Leftrightarrow w W_I \supseteq w^{\prime} W_{I^{\prime}}, \quad \text { for } w, w^{\prime} \in W \quad \text { and } I, I^{\prime} \subset \Pi .
$$
\end{cor}

\begin{defn}[Coxeter Complex defined via posets]
  \label{defn:CoxeterPosets}
  Let $(W, S)$ be a finite Coxeter system, with $ \operatorname{Card}(S) \geq 2$. The \emph{Coxeter poset} $X$ is the $W$-poset consisting of all left cosets
$$
\left\{w W_I: w \in W, \quad I \subset S\right\},
$$
ordered by the inverse inclusion. The action of $W$ on $X$ is given by left translation: for
$x \in W$, the action is
$$
w W_I \mapsto x w W_I \quad \textrm{for all } w \in W, \quad I \subset S .
$$
\end{defn}
We see that such defined poset satisfies Definition
\ref{defn:SimplicialComplex}, thus we can see $X$ as a complex. Moreover, it
gives us the same complex as Definition \ref{defn:CoxeterComplex} using
Corollary \ref{cor:cplxposet}.

\begin{thm}\label{thm:SphereHomology}
  Let $R$ be a root system in $V$, and $(W, S)$ the Coxeter system associated
  with $R$, and $S$ is the set of reflections associated with $\Delta$. Recall
  that the dimension of $V$ equals $\operatorname{S}=n \geq 2$. Let $X$ be the
  Coxeter poset associated with $(W,S)$. We have the
  following:
  \begin{enumerate}[(i)]
  \item The Weyl group acts on the unit sphere $S^{n-1}$, and there is a
    $W$-equivalent homeomorphism
    \[
|X| \cong S^{n-1}
\]
where $|X|$ is defined in Definition \ref{defn:TopSigma}.
  \item The homology representation of $W$ on $H_{*}(C_{\bullet} (X))$ over the
    rational field $\mathbb{Q}$ is as follows:
   \[
H_i(C_{\bullet}(X)) =
\begin{cases}
1_W , & \text{if } i=0, \\
\Sgn_{W} ,& \text{if } i=n-1, \\
0 ,& \text{otherwise.}
\end{cases}
\]
where the sign representation is defined by $\Sgn_W (w) = \det w$ for all $w
\in W$, and $1_{W} (w) =1$ for all $w \in W$.
  \end{enumerate}
\end{thm}
\begin{proof}
   \begin{enumerate}[(i)]
  \item   Since $W$ is a group of orthogonal transformations, $W$ clearly acts on the
   sphere $S^{n-1}$. We now give an outline of the proof that $|X| \cong
   S^{n-1}$. Let $\overline{C}$ be the chamber in $V$ defined by $\Delta$. It is
   easily verified, using the fact that $\Delta$ is a basis of $V$, that the
   hyperplanes $H_\alpha, \  \alpha \in \Delta$ are the walls of $\overline{C}$,
   in the sense that $H_\alpha \cap \overline{C}$ is contained in the boundary
   of $\overline{C}$, and generates the vector space $H_\alpha$. The same
   argument shows that each face $\overline{C} \cap H_{\alpha_{i_1}} \cap
   \cdots \cap H_{\alpha_{i_g}}$, where $I=\left\{\alpha_{i_1}, \ldots,
     \alpha_{i_g}\right\} \subset \Delta$, generates a subspace of $V$ of
   dimension $n-|I|$. Using the discussion in the previous section, it is then
   readily shown that the intersection $\overline{C} \cap S^{n-1}$ is
   homeomorphic to the underlying topological space of an (abstract)
   $(n-1)$-simplex $\sigma$. The vertices of $\sigma$ correspond to the
   intersections $C_I \cap S^{n-1}$, for $|I|=n-1, I \subset \Delta$, since for
   such $I$, each face $C_I$ is a half-line and intersects $S^{n-1}$ in a point.
   These points support a spherical simplex homeomorphic to $|\sigma |$.

   By Proposition \ref{prop:CoxeterComplexProperties}, $S^{n-1}$ is the union of
the intersections $\left\{w \overline{C} \cap S^{n-1}\right\}_{w \in W}$, and
their interiors partition the set $S^{n-1}-\bigcup_{\alpha \in \Delta}
H_\alpha$. By the first paragraph, each such intersection is homeomorphic to
$|\sigma|$. These intersections $\left\{w \overline{C} \cap S^{n-1}\right\}_{w
  \in W}$, and their faces $\left\{w C_I \cap S^{n-1}\right\}_{w \in W}$, where
$I \subset S$, form a $W$-poset under inclusion, with $W$-action given by
left translation. By Corollary \ref{cor:cplxposet}, this $W$-poset is isomorphic
to the Coxeter poset $X$. By  \cite[Proposition 66.1]{CurtisReinerVolII}, the simplicial complex $X$ is $W$-isomorphic to the barycentric subdivision of the simplicial complex whose geometric realization is $S^{n-1}$. Hence there is a $W$ equivariant homeomorphism $|X| \cong S^{n-1}$.
\item Since $|X| \cong S^{n-1}$. We know from the results on the homology of
  $(n-1)$-sphere that $H_i (C_{\bullet} (X)) = 0$ if $i \neq 0,n-1$, and $H_0
  (C_{\bullet} (X)) =H_{n-1} (C_{\bullet} (X)) \cong \mathbb{Q}$. From the
  definition of $H_0 (C_{\bullet} (X))$ we know that it affords the trivial
  representation. The Lefschetz character $\Lef$ of $H_{*} (C_{\bullet} (X))$ is
  given by
  \begin{equation}
    \label{eq:6628}
    \Lef = 1_W + (-1)^{n-1} \Tr ( \cdot, H_{n-1}).
  \end{equation}
Now let $s$ be any element in $S$. The fixed point $|X|^s \cong
(S^{n-1})^{s} \cong S^{n-2}$, whose usual Euler characteristic is $1 +
(-1)^{n-2}$. We know from Proposition \ref{prop:LefschetzComputation} that
\[
  \Lef (s)= 1 + (-1)^{n-2}.
\]
Comparing with \refeq{eq:6628}, we find $\Tr ( s, H_{n-1})=-1$ for
all $s \in S$. Thus $H_{n-1}(C_{\bullet} (X))$ affords the sign representation.
\end{enumerate}
\end{proof}

\subsection{Proof of Howlett-Lehrer's Theorem for characters}\label{subsec:ProofHLAppendix}
Let $C_{I_0}(I) $ be defined as the set $  \{ w   \in W \ | \ w I_0 \subset  \langle  I \rangle  \}$.
\begin{lem}[Howlett-Lehrer]
  \label{lem:HowlettLehrer1982Lem1A}
  The set $R_{I_0} = \{ wC_{I} \ | \ I \subset S, \ w \in C_{I_0}(I) \}$ is precisely the set of
  those $W$-regions contained in $I_0^{\perp}$.
\end{lem}

\begin{cor}
  \label{cor:HowlettLehrer1982Cor1A}
  The poset $X_{I_0^{\bot}} = \{ wW_{I} \ | \ I \subset S, \ w \in C_{I_0}(I)
  \}$ (abbreviated as $X_0$) defines a
  subcomplex of the Coxeter complex $X$ of $W$, which corresponds to the simplicial
  subdivision $\left\{ S(I_0^{\perp}) \cap wC_I \ | \  I \subset S, \ w \in
    C_{I_0}(I) \right\}$ of $S(I_0^{\perp})$.
\end{cor}
\begin{thm}[R.B. Howlett, G.I. Lehrer]
  \label{thm:HowlettLehrer1982A}
  Let $I_0$ be a fixed subset of $S$, and let $H$ be any subgroup of $N_{W}(W_{I_0})$. Let $\chi$ be a character of $H$, we have the following equation of characters:
  \begin{equation}
    \label{eq:HowlettLehrerCor1A}
    \sum_{I \subset S} (-1)^{|I|} \sum_{w \in W_I  \backslash C_{I_0}(I) /H } \Ind^{H}_{H \cap  W_I^w} (\Res^{H}_{H \cap  W_I^{w}} (\chi))  = \widehat{\chi}:= (-1)^{|I_0|} (-1)^{\ell_{I_0^{\perp}}(-)} \chi
  \end{equation}
where $I_0^{\perp}$ is its orthogonal
  complement of $I_0$.
\end{thm}
We give now give a detailed proof of this theorem following \cite{HowlettLehrer1982}.
\begin{proof}
  We first prove for any subgroup $H \leqslant N_W\left(W_L\right)$,
  \begin{equation}
    \label{eq:HLTriv}
    \sum_{I \subset S}(-1)^{ |I| } \sum_{w \in  W_I \backslash C_{I_0}(I) /H }
\operatorname{Ind}_{H \cap W_I^w}^H(\operatorname{Res}_{H \cap W_I^w}^H 1_H)=(-1)^{|I_0|} (-1)^{\ell_{I_0^{\bot}} (-)}1_H.
  \end{equation}
Here the second summation is over a set of representatives for the $\left(W_I, H
\right)$-double cosets contained in $C_{I_0}(I)$, and the notation $1_H$ denotes
the  character of trivial representation of $H$.

Let $H$ be any subgroup of $N_W(W_{I_0})$, then $H$ acts as a group of
  orthogonal transformation of $I_0^{\bot}$, hence as  a group of homeomorphisms
  of the unit sphere $S (I_{0}^{\bot})$ in $I_{0}^{\bot}$. Applying Theorem
  \ref{thm:SphereHomology}, we obtain the Lefschetz number $\Lef_{I_0^{\bot}}$:
 \begin{equation}
  \label{eq:LefI_0h0}
\Lef_{I_0^{\bot}} (h)= \sum_{i=0}^{\infty} (-1)^{i}\Tr (h, H_i
(S(I_0^{\bot})))=1+(-1)^{n-|I_0|-1} \det w|_{I_0^{\bot}} .
 \end{equation}
where $n= \dim V = \operatorname{Card}(S)$, and $ \det w|_{I_0^{\bot}}=  (-1)^{\ell_{I_0^{\bot}} (w)}$. From Proposition
\ref{prop:LefschetzComputation}, we have
\[
  \Lef_{I_0^{\bot}} (h) = \chi_{I_0^{\bot}} (X_0^h) = \sum_{i=0}^{\infty} (-1)^i \dim
C_r (X_0^h).
\]
From Definition \ref{defn:CoxeterPosets} and its relation (Corollary
\ref{cor:cplxposet}) with the $W$-posets $wC_I$, we know $X_0^h$ is the $W$-poset formed by
\[
\{ wW_{I} \ | \ I \subseteq \Delta, \  w \in C_{I_0}(I), \ hwW_I=wW_I \}.
\]
The $i$-simplex $S(I_0^{\bot} \cap wC_I)$ in the simplicial subdivision of
$S(I_0^{\bot})$ corresponds to the $i$-chain of the following base
\[
\{ wW_{I} \ | \ I \subseteq \Delta, \  w \in C_{I_0}(I), \ hwW_I=wW_I , \
\operatorname{Card}(I) = n-1-i \}.
\]
Thus we know
\begin{equation}
  \label{eq:LefI_0h1}
  \Lef_{I_0^{\bot}} (h) = \sum_{I \subset S, \ I \neq S} (-1)^{n-1-|I|}  n_{I_0, I}(h),
\end{equation}
where $n_{I_0, I}(h)$ is the number of $wW_I$ fixed by the left action of $h$ in
the subcomplex $X_0$, \emph{i.e.}
\begin{equation}\label{eq:nI0I1}
    \begin{aligned}
 n_{I_0, I}(h) & = \operatorname{Card}(\{ x W_{I} \ | \ I \subseteq \Delta, \  x
  \in C_{I_0}(I), \ hx W_I=xW_I \}) \\
                 &   = \sum_{w \in  H \backslash C_{I_0}(I) / W_I } \operatorname{Card} ( \{ x \in   H w W_I /W_I \ | \ h x W_I =xW_I \} ) \\
                 &  = \sum_{w \in  H \backslash C_{I_0}(I) / W_I } \operatorname{Card} ( \{ x  \in H / (H \cap {}^{w} W_I) \ | \ h x  (H \cap {}^{w}W_I)=x(H \cap {}^{w}W_I) \}) .
    \end{aligned}
  \end{equation}
The last equations follows from that the action of $h$ on the set of cosets $HwW_{I} /W_I$ is
  equivalent to the action of $h$ on the set of cosets $H / (H \cap {}^{w}W_I)$.
Recall that if $G$ is any finite group and $K$ is a subgroup of $G$, we define
the operations of restriction and induction of class functions in the usual way.
The induction to $G$ of a class function $ \phi$ of $K$ is given
\begin{equation}
  \label{eq:classfuncInd}
  (\Ind_K^{G} \phi ) (g)=\frac{1}{|K|} \sum_{x \in G} \phi\left(x g x^{-1}\right) = \phi (g) \frac{1}{|K|}  \sum_{x \in G} 1 = \phi (g) (\Ind_K^G 1_K )(g),
\end{equation}
summed over those elements $x \in G$ for which $x g x^{-1} \in K$. Especially,
for $K=H \cap {}^{w}W_I$, $G=H$, we have
\[
 \operatorname{Card} ( \{ x  \in H / (H \cap {}^{w}W_I) \ | \ h x  (H \cap
 {}^{w} W_I)=x(H \cap {}^{w} W_I) \}) = (\operatorname{Ind}_{H \cap W_I^w}^H 1_{H \cap W_I^w})(h).
\]
Thus from \refeq{eq:nI0I1} we know
\begin{equation}
  \label{eq:nI0I2}
 n_{I_0, I}(h) = \sum_{w \in  W_I \backslash C_{I_0}(I) / H } (\operatorname{Ind}_{H \cap W_I^w}^H 1_{H \cap W_I^w})(h).
\end{equation}
Substitute \refeq{eq:nI0I2} into \refeq{eq:LefI_0h1} and then compare \refeq{eq:LefI_0h1} and \refeq{eq:LefI_0h0}, we have the following result
\begin{equation}
  \label{eq:HLTriv2}
  \sum_{I \subset S}(-1)^{ |I| } \sum_{w \in  W_I \backslash C_{I_0}(I) /H }
\operatorname{Ind}_{H \cap W_I^w}^H(\operatorname{Res}_{H \cap W_I^w}^H 1_H)=(-1)^{|I_0|} (-1)^{\ell_{I_0^{\bot}} (-)}1_H.
\end{equation}
Using \refeq{eq:classfuncInd}, we deduce \refeq{eq:HowlettLehrerCor1A} just by
multiplying \refeq{eq:HLTriv2} ($H= W( \Lambda )$) by $\chi$, an element in the character
ring of $H$.
\end{proof}

\end{appendices}


\end{document}